\documentclass{article}
\usepackage{amsmath}
\usepackage{amsfonts}
\usepackage{amssymb}
\usepackage{graphicx}
\usepackage{xcolor}

\newtheorem{theorem}{Theorem}[section]

\newtheorem{example}[theorem]{Example}

\newtheorem{proposition}[theorem]{Proposition}
\newtheorem{remark}[theorem]{Remark}

\newenvironment{proof}[1][Proof]{\noindent\textbf{#1.} }{\ \rule{0.5em}{0.5em}}
\numberwithin{equation}{section}

\begin{document}

\title{Unified a-priori estimates for minimizers under $p,q-$growth and
exponential growth}
\author{Paolo Marcellini$^{1}$, Antonella Nastasi$^{2}$, Cintia Pacchiano
Camacho$^{3}$ \\
\medskip \\
{\normalsize $^{1}$Dipartimento di Matematica e Informatica
\textquotedblleft U. Dini\textquotedblright , Universit\`{a} di Firenze}\\
{\normalsize Viale Morgagni 67/A, 50134 - Firenze, Italy}\\
{\normalsize paolo.marcellini@unifi.it}\\
{\normalsize $^{2}$Department of Engineering, University of Palermo}\\
{\normalsize Viale delle Scienze, 90128 - Palermo, Italy}\\
{\normalsize antonella.nastasi@unipa.it}\\
{\normalsize $^{3}$ Theoretical Sciences Visiting Program,}\\{\normalsize Okinawa Institute of Science and Technology Graduate
University}\\
{\normalsize Onna, 904-0495, Japan}\\
{\normalsize cintia.pacchiano@oist.jp }}
\date{}
\maketitle

\begin{abstract}


We propose some \textit{general growth conditions} on the function $%
f=f\left( x,\xi \right) $, including the so-called \textit{natural growth},
or \textit{polynomial}, or $p,q-$\textit{growth conditions}, or even \textit{%
exponential growth}, in order to obtain that any local minimizer of the
energy integral $\;\int_{\Omega }f\left( x,Du\right) dx\,$\ is \textit{%
locally Lipschitz continuous} in $\Omega $. In fact this is the fundamental
step for further regularity: the \textit{local boundedness of the gradient}
of any Lipschitz continuous local minimizer \textit{a-posteriori} makes
irrelevant the behavior of the integrand $f\left( x,\xi \right) $ as $%
\left\vert \xi \right\vert \rightarrow +\infty $; i.e., the \textit{general
growth conditions} a posteriori are reduced to a standard growth, with the
possibility to apply the classical regularity theory. In other words, we
reduce some classes of \textit{non-uniform} elliptic variational problems to
a context of \textit{uniform} ellipticity.\newpage
\end{abstract}

\tableofcontents

\bigskip

\bigskip

\bigskip



\section{Introduction}

Let $\Omega $ be an open set in $\mathbb{R}^{n}$, $n\geq 2$, and $x\in
\Omega \subset \mathbb{R}^{n},\xi \in \mathbb{R}^{n}$ be generic vectors.
Let $f=f\left( x,\xi \right) $, $f:\Omega \times \mathbb{R}^{n}\rightarrow 
\mathbb{R}$ be a continuous function in $\Omega \times \mathbb{R}^{n}$. The
related \textit{energy integral} where to look for local minimizers is%
\begin{equation}
v\;\rightarrow \;\int_{\Omega }f\left( x,Dv\right) \,dx\,.
\label{energy integral}
\end{equation}%
As well known a \textit{local minimizer} of the energy integral (\ref{energy
integral}) is a Sobolev function $u$ such that $f\left( x,Du\right) \in L_{%
\mathrm{loc}}^{1}\left( \Omega \right) $ and $\int_{\Omega }f\left(
x,Du\right) \,dx\leq \int_{\Omega }f\left( x,Du+D\varphi \right) \,dx$ for
every Sobolev function $\varphi $ whose support is contained in $\Omega $.

In this manuscript we allow examples of \textit{anisotropic polynomial growth%
}, a class of variational energy integrals which may have singular (even not
bounded) local minimizers. Such as for instance 
\begin{equation*}
f\left( x,Dv\right) =\sum\limits_{i,j=1}^{n}a_{ij}\left( x\right)
v_{x_{i}}u_{x_{j}}+\left\vert v_{x_{n}}\right\vert ^{q},
\end{equation*}%
with $q\geq 2$ and $\left( a_{ij}\right) $ being an $n\times n$ a positive
definite matrix of locally Lipschitz continuous coefficients in $\Omega $.
We also allow \textit{exponential growth} such as, for example, 
\begin{equation*}
f\left( x,Dv\right) =\exp (a\left( x\right) \left\vert Dv\right\vert ^{2}).
\end{equation*}%
For these examples \textit{ellipticity} holds, but \textit{not uniform
ellipticity} as in the statement (\ref{uniform ellipticity conditions})
below.

Our aim is to propose some \textit{general growth conditions} on the
function $f=f\left( x,\xi \right) $ (including the previous polynomial and
exponential examples) in order to obtain that any \textit{local minimizer }%
is \textit{locally Lipschitz continuous} in $\Omega $. In fact this one is
the fundamental property for a local minimizer for further regularity: the 
\textit{local boundedness of the gradient} of any Lipschitz continuous local
minimizer \textit{a-posteriori} makes irrelevant the behavior of the
integrand $f\left( x,\xi \right) $ as $\left\vert \xi \right\vert
\rightarrow +\infty $; i.e., the \textit{general growth conditions} a
posteriori are reduced to a standard growth, with the possibility to apply
the classical regularity theory, valid in the uniform elliptic contexts and
obtain, when possible, $C^{1,\alpha }-$regularity as a consequence. In fact,
having in force the local Lipschitz continuity considered in this
manuscript, \textit{also the }$C^{1,\alpha }$\textit{\ regularity can be
deduced} under the same assumptions made in the context of natural growth,
for instance as in Lady\v{z}enskaja-Ural'ceva \cite[Chapter 4, Section 61]%
{Ladyzhenskaja-Uralceva 1968} or in Giusti \cite[Sections 8.6 and 8.8]%
{Giusti 2003 book}; see the $p,q-$growth cases in \cite[Section 7, Theorem D]%
{Marcellini ARMA 1989} and in \cite[Corollary 2.2]{Marcellini 1991}. See
also the recent non-uniformly elliptic approach and gradient H\"{o}lder
continuity by DeFilippis-Mingione \cite[Section 6.4]{DeFilippis-Mingione Inv
2023}.

The \textit{convexity} with respect to the gradient variable $\xi \in 
\mathbb{R}^{n}$ is the classical assumption on the function $f=f\left( x,\xi
\right) $ when we consider existence of minimizers, once the class of
competing $v$ in (\ref{energy integral}) is fixed at the boundary $\partial
\Omega $. The corresponding convexity condition for a $C^{2}-$function $\xi
\rightarrow f\left( x,\xi \right) $, as well known, is related to the $%
n\times n$ matrix of second derivatives $(f_{\xi _{i}\xi _{j}})_{n\times n}$%
\ and it is the \textit{positivity of the quadratic form} $\lambda \in 
\mathbb{R}^{n}\rightarrow \sum_{i,j}\,f_{\xi _{i}\xi _{j}}\left( x,\xi
\right) \lambda _{i}\lambda _{j}$. When a \textit{strict qualified positivity%
} is required, we have the \textit{ellipticity condition} $%
\sum_{i,j}\,f_{\xi _{i}\xi _{j}}\left( x,\xi \right) \lambda _{i}\lambda
_{j}\geq m\left\vert \lambda \right\vert ^{2}$, valid for all $\lambda ,\xi
\in \mathbb{R}^{n}$, $x\in \Omega $ and for some positive constant $m$. More
precisely, often it is useful to emphasize the dependence of $m$ on the
gradient variable $\xi $; in fact more properly in general the constant $m$
is replaced by a function of $\xi $. In this manuscript we assume the
following \textit{ellipticity condition} 
\begin{equation*}
\sum_{i,j}\,f_{\xi _{i}\xi _{j}}\left( x,\xi \right) \lambda _{i}\lambda
_{j}\geq g_{1}\left( \left\vert \xi \right\vert \right) \left\vert \lambda
\right\vert ^{2},\;\;\;\;\;\forall \;\lambda ,\xi \in \mathbb{R}^{n},x\in
\Omega ,
\end{equation*}%
where $g_{1}:\left[ 0,+\infty \right) \rightarrow \left[ 0,+\infty \right) $
is a nonnegative increasing function. Less standard is the bound from above
for the $n\times n$ matrix of second derivatives $(f_{\xi _{i}\xi
_{j}})_{n\times n}$. The classical case is the \textit{uniformly elliptic}
one. This happens when a similar bound exists from above too: 
\begin{equation}
\text{\textit{uniformly ellipticity}:}\;\;\;g_{1}\left( \left\vert \xi
\right\vert \right) \left\vert \lambda \right\vert ^{2}\leq
\sum_{i,j=1}^{n}\,f_{\xi _{i}\xi _{j}}\left( x,\xi \right) \lambda
_{i}\lambda _{j}\leq Mg_{1}\left( \left\vert \xi \right\vert \right)
\left\vert \lambda \right\vert ^{2},  \label{uniform ellipticity conditions}
\end{equation}%
for all $\lambda ,\xi \in \mathbb{R}^{n}$, $x\in \Omega $, and for a
constant $M\geq 1$. The main example is the $p-$Laplacian; i.e., when the
integral in (\ref{energy integral}) is the $p-$\textit{Dirichlet integral},
with $f\left( \xi \right) =\left\vert \xi \right\vert ^{p}$. In this case a
computation shows 
\begin{equation*}
\sum_{i,j=1}^{n}\,f_{\xi _{i}\xi _{j}}\left( \xi \right) \lambda _{i}\lambda
_{j}=p\,[\left\vert \xi \right\vert ^{2}\left\vert \lambda \right\vert
^{2}+\left( p-2\right) (\sum_{i=1}^{n}\xi _{i}\lambda
_{i})^{2})]\,\left\vert \xi \right\vert ^{p-4}
\end{equation*}%
and thus the validity of the uniformly elliptic estimates 
\begin{equation*}
p-\text{\textit{Laplacian}:}\;\;\;p\left\vert \xi \right\vert
^{p-2}\left\vert \lambda \right\vert ^{2}\leq \sum_{i,j}\,f_{\xi _{i}\xi
_{j}}\left( \xi \right) \lambda _{i}\lambda _{j}\leq p\left( p-1\right)
\left\vert \xi \right\vert ^{p-2}\left\vert \lambda \right\vert ^{2},
\end{equation*}%
for all $\lambda ,\xi \in \mathbb{R}^{n}$, $x\in \Omega $, when $p\geq 2$;
otherwise, if $1<p<2$, it is necessary to interchange the first and the last
sides.

Let us go back to what we said at the very beginning of this introduction.
It is known that perturbations of the $p-$\textit{Dirichlet integral }$%
f\left( \xi \right) =\left\vert \xi \right\vert ^{p}$, $\xi =\left( \xi
_{1},\xi _{2},\ldots ,\xi _{n}\right) $, for instance with either $f\left(
\xi \right) =\sum_{i=1}^{n}\,\left\vert \xi _{i}\right\vert ^{p_{i}}$, or $%
f\left( \xi \right) =\left\vert \xi \right\vert ^{p}+\left\vert \xi
_{n}\right\vert ^{q}$ with $p<q$, give rise to energy integrals as in (\ref%
{energy integral}) which may admit not smooth local minimizers, even
unbounded local minimizers; see for instance \cite[Theorem 6.1]{Marcellini
1991}. In these cases the uniform ellipticity condition (\ref{uniform
ellipticity conditions}) does not hold. In the non-uniformly elliptic case
regularity of local minima may not be true.

Another non-uniformly elliptic case which we like to emphasize is the
exponential one, for instance with $f\left( \xi \right) =e^{\left\vert \xi
\right\vert ^{2}}$. In this example a local minimizer of the corresponding
energy integral in (\ref{energy integral}) a-priori not necessarily
satisfies the Euler's differential equation. The reason is due to the fact
that if $u$ is a local minimizer, with $\int_{\Omega }e^{\left\vert
Du\right\vert ^{2}}\,dx<+\infty $, then to obtain the first variation and
the Euler's equation, for any fixed test function $\varphi $ we need to
compute the limit, as $h\rightarrow 0$ ($h\in \mathbb{R}$, $h\neq 0$), of
the difference quotient 
\begin{equation*}
\tfrac{1}{h}(\,\int_{\Omega }e^{\left\vert D\left( u+h\varphi \right)
\right\vert ^{2}}\,dx-\int_{\Omega }e^{\left\vert Du\right\vert
^{2}}\,dx\,)\,;
\end{equation*}%
if $\varphi =\eta u$ and $\eta =1$ on a subset of $\Omega $ (as usually it
happens in the regularity approach), i.e. if $\varphi =u$ and $h>0$ on a
subset $\Omega ^{\prime }$ of $\Omega $, then $\int_{\Omega ^{\prime
}}e^{\left\vert D\left( u+h\varphi \right) \right\vert
^{2}}\,dx=\int_{\Omega ^{\prime }}e^{\left( 1+h\right) ^{2}\left\vert
Du\right\vert ^{2}}\,dx$; thus the summability of $e^{\left\vert
Du\right\vert ^{2}}$ in general does not implies the summability of $%
e^{\left( 1+h\right) ^{2}\left\vert Du\right\vert ^{2}}$ because of the
factor $\left( 1+h\right) ^{2}>1$ and is not possible to go directly to the
Euler's equation. An energy integral of exponential type in general is more
difficult to be studied from the regularity point of view. The exponential
energy integrals are not uniformly elliptic and sometime they need to be
treated with appropriate specific techniques (see Cellina-Staicu \cite%
{Cellina-Staicu 2018}).

In this manuscript we propose a unified approach to regularity for
minimizers of nonuniformly elliptic energy integral, which include \textit{%
exponential growth} and $p,q-$\textit{growth}. This last case is related to
the functions $g_{1}$,$g_{2}$ below in (\ref{nonuniformly elliptic
conditions}), when they are powers with exponents $p,q$; more precisely, $%
g_{1}$,$g_{2}$ being compared with the second derivatives of $f$ with
respect to the gradient variable $\xi $, the powers to compare $g_{1}$,$%
g_{2} $ respectively are $\left\vert \xi \right\vert ^{p-2}$ and $%
1+\left\vert \xi \right\vert ^{q-2}$, or $(1+\left\vert \xi \right\vert
^{2})^{\left( q-2\right) /2}$. More in general we consider the 
\begin{equation}
\text{\textit{non-uniformly elliptic case}:}\;\;g_{1}\left( \left\vert \xi
\right\vert \right) \left\vert \lambda \right\vert ^{2}\leq
\sum_{i,j=1}^{n}\,f_{\xi _{i}\xi _{j}}\left( x,\xi \right) \lambda
_{i}\lambda _{j}\leq g_{2}\left( \left\vert \xi \right\vert \right)
\left\vert \lambda \right\vert ^{2},
\label{nonuniformly elliptic conditions}
\end{equation}%
for all $\lambda ,\xi \in \mathbb{R}^{n}$, $x\in \Omega $, with $g_{1},g_{2}:%
\left[ 0,+\infty \right) \rightarrow \left[ 0,+\infty \right) $ nonnegative
increasing functions. Of course it is necessary to impose conditions on
these functions $g_{1},g_{2}$ in order to obtain a-priori estimates for
regularity. Full details are stated in Section \ref{Section of the statement}%
, together with the assumption of the local Lipschitz continuity in $\Omega
\times \mathbb{R}^{n}$ of the gradient with respect to $\xi \in \mathbb{R}%
^{n}$ of $f=f\left( x,\xi \right) $; i.e. the local Lipschitz continuity of $%
D_{\xi }f\left( x,\xi \right) =\left( f_{\xi _{i}}\left( x,\xi \right)
\right) _{i=1,2,\ldots ,n}$.

Nowadays many papers deal with regularity for non-uniformly elliptic
problems. Limited to the more recent literature, for interior regularity we
refer to Colombo-Mingione \cite{Colombo-Mingione 2015a}-\cite%
{Colombo-Mingione 2016}, Baroni-Colombo-Mingione \cite%
{Baroni-Colombo-Mingione 2018}, Eleuteri-Marcellini-Mascolo \cite%
{Eleuteri-Marcellini-Mascolo 2020}, Bousquet-Brasco \cite{Bousquet-Brasco
2020}, DeFilippis-Mingione \cite{DeFilippis-Mingione 2020}-\cite%
{DeFilippis-Mingione Inv 2023}, Mingione-R\u{a}du\-lescu \cite%
{Mingione-Radulescu 2021}, B\"{o}gelein-Duzaar-Giova-Passarelli-Scheven \cite%
{Boegelein-Duzaar-Giova-Passarelli-Scheven 2023}; see also \cite%
{Cupini-Marcellini-Mascolo 2018}-\cite{Cupini-Marcellini-Mascolo-Passarelli
2023} and \cite{Duzgun-Marcellini-Vespri 2014 Parma}-\cite%
{Eleuteri-Perrotta-Treu 2025}. About recent boundary regularity under
general growth conditions we mention Cianchi-Maz'ya \cite{Cianchi-Mazya 2011}%
,\cite{Cianchi-Mazya 2014}, B\"{o}gelein-Duzaar-Marcellini-Scheven \cite%
{Boegelein-Duzaar-Marcellini-Scheven JMPA 2021}, DeFilippis-Piccinini \cite%
{Defilippis-Piccinini 2022-2023}. For \textit{Orlicz-Sobolev spaces,
variable exponents} and \textit{double phase} see
Diening-Harjulehto-Hasto-Ruzicka \cite{Diening-Harjulehto-Hasto-Ruzicka 2011}%
, Chlebicka \cite{Chlebicka 2018}, Chlebicka-DeFilippis \cite%
{Chlebicka-DeFilippis 2019}, Byun-Oh \cite{Byun-Oh 2020}, Ragusa-Tachikawa 
\cite{Ragusa-Tachikawa 2020}, H\"{a}st\"{o}-Ok \cite{Hasto-Ok 2022},
Crespo-Blanco-Gasi\'{n}ski-Winkert \cite{Crespo-Blanco-Gasinski-Winkert 2024}%
. Higher integrability and stability of $p,q-$quasi-minimizers in double
phase problems by Kinnunen-Nastasi-Pacchiano Camacho \cite%
{Kinnunen-Nastasi-Pacchiano Camacho 2024} and Nastasi-Pacchiano Camacho \cite%
{Nastasi-Pacchiano Camacho 2021}-\cite{Nastasi-Pacchiano Camacho 2024}. 
\textit{Quasiconvex integrals} of the calculus of variations in \cite%
{Boegelein-Dacorogna-Duzaar-Marcellini-Scheven 2020},\cite%
{Dacorogna-Marcellini 1998},\cite{Marcellini 1984}, and about \textit{%
partial regularity} by Schmidt \cite{Schmidt 2009}, DeFilippis \cite{De
Filippis JMPA 2022}, DeFilippis-Stroffolini \cite{De Filippis-Stroffolini
2023}, Gmeineder-Kristensen \cite{Gmeineder-Kristensen 2022}.

Particularly related to this manuscript, the literature on regularity for 
\textit{non-uniformly elliptic problems}, which consider at the same time 
\textit{exponential growth}, is less wide. It starts from \cite[Section 6]%
{Marcellini 1993} for a class of non-uniformly elliptic equations including 
\textit{``slow"} exponential growth. The first regularity result specific
for local minimizers in $u\in W_{\mathrm{loc}}^{1,2}\left( \Omega \right) $
such that $f\left( Du\right) \in L_{loc}^{1}\left( \Omega \right) $,
possibly with exponential growth, is in \cite{Marcellini JOTA 1996}, which
deals with energy integrals as in (\ref{energy integral}) with integrand $%
f=f\left( \xi \right) $ independent of $x$, who partially inspired our
research here. In the same year an approach to \textit{the vector-valued case%
} was introduced in \cite{Marcellini Everywhere 1996}, later generalized by
Marcellini-Papi \cite{Marcellini-Papi 2006}, both again related to $%
f=f\left( \xi \right) $ independent of $x$. The first extensions to $%
f=f\left( x,\xi \right) $ are due to Mascolo-Migliorini \cite%
{Mascolo-Migliorini 2003} and DiMarco-Marcellini \cite{DiMarco-Marcellini
2020}, who treated \textit{the vector-valued case}, with $f:=g\left(
x,\left\vert \xi \right\vert \right) $ of \textit{Uhlenbeck-type} \cite%
{Uhlenbeck 1977}, i.e. depending on the modulus $\left\vert \xi \right\vert $%
. Finally we mention Beck-Mingione \cite{Beck-Mingione 2019}, who studied
energy integrals of the form $\int_{\Omega }\left\{ g\left( \left\vert
Du\right\vert \right) +h\left( x\right) \cdot u\right\} dx$ and they
considered some sharp assumptions on the function $h\left( x\right) $, of
the type $h\in L\left( n,1\right) \left( \Omega ;\mathbb{R}^{m}\right) $ in
dimension $n>2$ (i.e., $\int_{0}^{+\infty }\mathrm{{meas}\left\{ x\in \Omega
:\left\vert h\left( x\right) \right\vert >\lambda \right\} ^{1/n}d\lambda
<+\infty }$; note that $L^{n+\varepsilon }\subset L\left( n,1\right) \subset
L^{n}$), or $h\in L^{2}\left( \log L\right) ^{\alpha }\left( \Omega ;\mathbb{%
R}^{m}\right) $ for some $\alpha >2$ when $n=2$. Beck-Mingione obtained the
local boundedness of the gradient $Du$ allowing exponential growth too;
however the function $g\left( \left\vert \xi \right\vert \right) $ is
assumed to be depending on the modulus $\left\vert \xi \right\vert $ of $\xi 
$ and independent of $x$. Therefore the a-priori estimates of Theorem \ref%
{Lipschitz continuity result} below is new with respect to the known
literature, and in particular with respect to the quoted references.

\section{Statement of the main results\label{Section of the statement}}

We consider a function $f=f\left( x,\xi \right) $, $f:\Omega \times \mathbb{R%
}^{n}\rightarrow \mathbb{R}$, whose gradient with respect to the $\xi -$%
variable $D_{\xi }f\left( x,\xi \right) =\left( f_{\xi _{i}}\left( x,\xi
\right) \right) _{i=1,2,\ldots ,n}$ is locally Lipschitz continuous in $%
\Omega \times \mathbb{R}^{n}$. We assume that the following second partial
derivatives of $f$ (which exist in the $W_{\mathrm{loc}}^{1,\infty }-$sense)
satisfy the following \textit{ellipticity} and \textit{growth conditions:}
for every open set $\Omega ^{\prime }$ compactly contained in $\Omega $
there exist nonnegative increasing functions $g_{1},g_{2},g_{3}:\left[
0,+\infty \right) \rightarrow \left[ 0,+\infty \right) $, not identically
equal to zero, such that 
\begin{equation}
g_{1}\left( \left\vert \xi \right\vert \right) \,\left\vert \lambda
\right\vert ^{2}\leq \sum_{i,j}\,f_{\xi _{i}\xi _{j}}\left( x,\xi \right)
\lambda _{i}\lambda _{j}\leq g_{2}\left( \left\vert \xi \right\vert \right)
\,\left\vert \lambda \right\vert ^{2},  \label{ellipticity}
\end{equation}

\begin{equation}
\sum_{i}\,\left\vert f_{\xi _{i}x_{k}}\left( x,\xi \right) \right\vert \leq
g_{3}\left( \left\vert \xi \right\vert \right) ,  \label{growth conditions A}
\end{equation}%
for all $\xi ,\lambda \in \mathbb{R}^{n}$, $k=1,2,\ldots ,n$, a.e.\ $x\in
\Omega ^{\prime }$. These functions $g_{1},g_{2},g_{3}$ are related to each
other by 
\begin{equation}
\left( g_{2}\left( t\right) \right) ^{2\gamma -1}\,t^{2}\leq M\left\{
1+\int_{0}^{t}\sqrt{g_{1}\left( s\right) }ds\right\} ^{\alpha },  \label{11M}
\end{equation}%
\begin{equation}
\left( g_{2}\left( \left\vert \xi \right\vert \right) \right) ^{2\gamma
-1}\,\left\vert \xi \right\vert ^{2\gamma }\leq M\left\{ 1+f\left( x,\xi
\right) \right\} ^{\beta },  \label{12M}
\end{equation}

\begin{equation}
g_{3}\left( t\right) \underset{}{\leq }M\left( 1+t^{\gamma }\right) \left(
g_{1}\left( t\right) \right) ^{\frac{1}{2}}\left( g_{2}\left( t\right)
\right) ^{\gamma -\frac{1}{2}},  \label{Assumption 3}
\end{equation}%
for a positive constant $M=M\left( \Omega ^{\prime }\right) $ and for all $%
\xi \in \mathbb{R}^{n}$ and $t\geq 0$. Finally, the exponents $\alpha ,\beta
,\gamma $ satisfy the bounds 
\begin{equation}
2\leq \alpha <2^{\ast }-2\left( \gamma -1\right) \,,
\label{condition on alpha}
\end{equation}%
\begin{equation}
1\leq \beta <\frac{2\left( \alpha +2\gamma -2\right) }{n\left( \alpha
+2\gamma -4\right) }\,,  \label{condition on beta}
\end{equation}%
for some $\gamma \geq 1$. Here, as usual, we denote by $2^{\ast }$ the
Sobolev exponent defined by $2^{\ast }=\frac{2n}{n-2}$ if $n>2$, while $%
2^{\ast }$ is a fixed real number large enough when $n=2$; precisely such
that $\alpha <2^{\ast }$ and $\gamma <\frac{2^{\ast }-\alpha +2}{2}$.

\begin{remark}
\label{alphabetagamma} The conditions (\ref{condition on alpha}),(\ref%
{condition on beta}) on the exponents $\alpha $,$\beta $ and $\gamma \geq 1$
are satisfied if we choose, for instance, 
\begin{equation}
\beta =\frac{\alpha }{2}+\delta \,,\;\;\;\text{and}\;\;\;\gamma =1+\delta \,,
\label{alphabetagamma values}
\end{equation}%
with $2\leq \alpha <2^{\ast }-2\left( \gamma -1\right) =2^{\ast }-2\delta $, 
$1\leq \beta <1+\frac{2}{n}$ and $0\leq \delta <\frac{4}{n(n-2)}$. Indeed,
with the choice $\gamma =1+\delta $, (\ref{condition on beta}) assumes the
form 
\begin{equation*}
1\leq \tfrac{\alpha }{2}+\delta =\tfrac{\alpha +2\delta }{2}<\tfrac{2(\alpha
+2\delta )}{n(\alpha -2+2\delta )},
\end{equation*}%
which gives $\alpha -2+2\delta <\frac{4}{n}$, that is $\beta <1+\frac{2}{n}$%
. Finally, we must satisfy the $\alpha -$bound $\alpha <2^{\ast }-2\delta $,
i.e. $0\leq \delta <\frac{2^{\ast }-\alpha }{2}$. Since $\alpha =2(\beta
-\delta )\leq 2\beta <2+\frac{4}{n}$, we obtain the sufficient bound for $%
\delta $ 
\begin{equation*}
0\leq \delta <\tfrac{2^{\ast }-2-\frac{4}{n}}{2}\;\underset{\text{if }n>2}{=}%
\;\;\tfrac{4}{n(n-2)}\,.
\end{equation*}

In the particular case $\gamma =1$ conditions (\ref{condition on alpha}),(%
\ref{condition on beta}) simplify into $2\leq \alpha <2^{\ast }$ and $1\leq
\beta <\frac{2\alpha }{n\left( \alpha -2\right) }$. For instance, for the so
called natural growth conditions with, up to multiplicative constants, $%
g_{1}\left( t\right) =t^{p-2}$, $g_{2}\left( t\right) =\left( 1+t\right)
^{p-2}$ for some $p\geq 2$, (\ref{11M}) becomes 
\begin{equation}
\left( 1+t\right) ^{p-2}\,t^{2}\leq M\left\{ 1+\int_{0}^{t}s^{\frac{p}{2}%
-1}ds\right\} ^{\alpha }=M_{1}\left( 1+t^{\frac{\alpha }{2}p}\right) \,
\label{p,p-growth}
\end{equation}%
and it is satisfied for a constant $M_{1}$ and for all $t\geq 0$ by choosing 
$\alpha =2$. Similarly (\ref{12M}) holds for $g_{2}\left( t\right) =\left(
1+t\right) ^{p-2}$ and $f\left( x,\xi \right) \geq \mathrm{{const}%
\,\left\vert \xi \right\vert ^{p}}$ if $\beta =1$. This in the natural
growth conditions the simplest choice is $\alpha =2$, $\beta =1$ and $\gamma
=1$. Similar computations when $\gamma =1$ can be done for the $p,q-$growth
case, with $g_{1}\left( t\right) =t^{p-2}$, $g_{2}\left( t\right) =\left(
1+t\right) ^{q-2}$ and $f\left( \xi \right) \geq \mathrm{{const}\,\left\vert
\xi \right\vert ^{p}}$ for some exponents $q\geq p\geq 2$. To test (\ref{11M}%
) we change $p$ with $q$ in the left hand side of (\ref{p,p-growth}) and we
obtain $q\leq \frac{\alpha }{2}p$ . To test (\ref{12M}), we consider the
sufficient condition $\left( 1+\left\vert \xi \right\vert \right)
^{q-2}\,\left\vert \xi \right\vert ^{2}\leq M\left\{ 1+\left\vert \xi
\right\vert ^{p}\right\} ^{\beta }$, which holds with $q\leq \beta p$. In
this case the natural choice, as in (\ref{alphabetagamma values}), is to fix 
$\beta =\frac{\alpha }{2}$ when $\gamma =1$ (i.e. $\delta =0$), under the
constraint $\alpha -2<\frac{4}{n}$ or equivalently $\beta <1+\frac{2}{n}$,
which - written in the explicit form $\frac{q}{p}<1+\frac{2}{n}$ - is a
natural bound for $p,q-$growth variational problems without $x-$dependence
(see for instance \cite[Remark 2.1]{Marcellini JOTA 1996} or \cite%
{Colombo-Mingione 2015a},\cite{Cupini-Marcellini-Mascolo 2023},\cite%
{Eleuteri-Marcellini-Mascolo 2020}). Note however that in general the $x-$%
dependence and specifically the exponential growth need to choose the
parameter $\gamma $ strictly grater than $1$, for instance as in (\ref%
{alphabetagamma values}). See details in Section \ref{Section: Consequences
and examples}.
\end{remark}

We are ready to state in Theorem \ref{Lipschitz continuity result} below our
regularity result, whose proof, divided into several steps, is given in
Section \ref{Section of the proofs}. With \textit{regularity} of $u$ we mean 
\textit{interior regularity}, without fixing Dirichlet boundary conditions,
or without other types of boundary conditions. Precisely, we obtain local
Lipschitz continuity estimates for $u$ and also local estimates of the $%
L^{2}-$norm of the $n\times n$ matrix $D^{2}u$ of the second derivatives of
a local minimizer $u$ of the energy integral (\ref{energy integral}). A
local minimizer is a Sobolev function $u$ such that 
\begin{equation*}
\int_{\Omega ^{\prime }}f\left( x,Du\right) \,dx<+\infty \;\;\;\text{and}%
\;\;\;\int_{\Omega ^{\prime }}f\left( x,Du\right) \,dx\leq \int_{\Omega
^{\prime }}f\left( x,D\left( u+\psi \right) \right) \,dx\,,
\end{equation*}%
for every $\psi \in W^{1,2}\left( \Omega \right) $ with support in the open
set $\Omega ^{\prime }$ whose closure is contained in $\Omega $. Theorem \ref%
{Lipschitz continuity result} is a regularity result for \textit{smooth
local minimizers} $u$ of the energy integral (\ref{energy integral}). These
estimates are in fact \textit{a priori estimates}: they hold for a priori 
\textit{smooth} local minimizers; precisely, for local minimizers $u$ in the
class 
\begin{equation}
\left\{ u\in W^{2,2}\left( \Omega ^{\prime }\right) :\int_{\Omega ^{\prime
}}\left\vert Du\right\vert ^{2\gamma }\left( g_{2}(\left\vert Du\right\vert
)\right) ^{2\gamma -1}dx+\int_{\Omega ^{\prime }}g_{2}\left( \left\vert
Du\right\vert \right) \,\left\vert D^{2}u\right\vert ^{2}dx<+\infty \right\}
,  \label{a priori smooth class for local minimizers}
\end{equation}%
where these summability conditions are valid for all sets $\Omega ^{\prime }$
compactly contained in $\Omega $. As well known, the natural condition is to
impose \textit{finite energy }for the local minimizer $u$; i.e. to belong to
the class 
\begin{equation}
\left\{ u\in W^{1,1}\left( \Omega ^{\prime }\right) :\;\;\int_{\Omega
^{\prime }}f\left( x,Du\right) \,dx<+\infty \right\} \,,
\label{class for local minimizers of finite energy}
\end{equation}%
while (\ref{a priori smooth class for local minimizers}) is an a priori
proper condition useful to prove the estimates (\ref{gradient bound})-(\ref%
{bound on the second derivatives}) of Theorem \ref{Lipschitz continuity
result}.

The auxiliary assumption (\ref{a priori smooth class for local minimizers})
should be removed later. Usually, the main steps for regularity are the 
\textit{a-priori estimates}; the following steps are obtained through an 
\textit{approximation procedure}. A reference example of this method can be
found in the first Lipschitz continuity result under \textit{non-uniformly
elliptic} and \textit{general growth condition}, obtained in \cite[Step 5]%
{Marcellini JOTA 1996} for local minimizers, however for the simpler case $%
f\left( x,\xi \right) =f\left( \xi \right) $, i.e. with $f$ independent of $%
x $. A further approximation procedure to go from the a priori estimates to
full regularity was done in \cite{Eleuteri-Marcellini-Mascolo 2020}: also
this one is related to a special case; precisely with \textit{%
modulus-dependence} on the $\xi -$variable $f\left( x,\xi \right) =g\left(
x,\left\vert \xi \right\vert \right) $, named \textit{Uhlenbeck structure},
recalling the celebrated paper \cite{Uhlenbeck 1977} published in 1977 by
Karen Uhlenbeck.

Finally we observe that we do not assume \textit{structure conditions} on
the function $f:\Omega \times \mathbb{R}^{n}\rightarrow \mathbb{R}$. In
particular we do not assume the mentioned \textit{Uhlenbeck structure} with
the modulus dependence $f\left( x,\xi \right) =g\left( x,\left\vert \xi
\right\vert \right) $ on the gradient variable, not even the so-called 
\textit{double phase} structure, with $f=f\left( x,\xi \right) :=\left\vert
\xi \right\vert ^{p}+a\left( x\right) \left\vert \xi \right\vert ^{q}$,
where $a\left( x\right) $ is a nonnegative coefficient which is equal to
zero somewhere in $\Omega $, neither the more general $p,q-$\textit{growth
conditions}, which are more general than the double phase case, but which of
course do not include \textit{exponential growth}, a case on the contrary
that enters in our Lipschitz continuous a-priori estimates of Theorem \ref%
{Lipschitz continuity result}.

\begin{theorem}
\label{Lipschitz continuity result}Under the ellipticity and growth
conditions (\ref{ellipticity})-(\ref{condition on beta}), let $u$ be a local
minimizer of the energy integral (\ref{energy integral}) in the Sobolev
class (\ref{a priori smooth class for local minimizers}). Then some uniform
a-priori estimates hold for the $L_{\mathrm{loc}}^{\infty }-$norm of the
gradient $Du$ and for the $n\times n$ matrix $D^{2}u$ of the second
derivatives of $u$. Precisely, for every open set $\Omega ^{\prime }$
compactly contained in $\Omega $ there exist exponents $\theta _{1},\theta
_{3}>1$, $\theta _{2},\theta _{4}>0$ and a radius $R_{0}>0$ such that\textit{%
\ }%
\begin{equation}
\left\Vert Du\right\Vert _{L^{\infty }\left( B_{\rho };\mathbb{R}^{n}\right)
}^{2}\leq \tfrac{c}{\left( R-\rho \right) ^{\theta _{2}}}\left(
\int_{B_{R}}\left( 1+f\left( Du\right) \right) \,dx\right) ^{\theta _{1}}\,,
\label{gradient bound}
\end{equation}%
\begin{equation}
\int_{B_{\rho }}g_{1}\left( \left\vert Du\right\vert \right) \left\vert
D^{2}u\right\vert ^{2}dx\leq \tfrac{c}{\left( R-\rho \right) ^{\theta _{4}}}%
\left( \int_{B_{R}}\left( 1+f\left( Du\right) \right) \,dx\right) ^{\theta
_{3}}\,,  \label{bound on the second derivatives}
\end{equation}%
\textit{for every }$\rho $\textit{,} $R$\textit{\ with }$0<\rho <R\leq R_{0}$%
\textit{\ }and for a positive constant $c$, \textit{where }$B_{\rho }$, $%
B_{R}$\textit{\ are concentric balls contained in }$\Omega ^{\prime }$ with
respective radii $\rho $, $R$.
\end{theorem}

The constant $c$ in (\ref{gradient bound}),(\ref{bound on the second
derivatives}) depends on the dimension $n$ and on the data, but it is
independent of $u$ itself. An explicit analytic representation of the
exponent $\theta _{1}$ is given below in (\ref{52M}), as well as the
analytic expression of $\theta _{2}$, $\theta _{3}$, $\theta _{4}$ in the
proof of Theorem \ref{Lipschitz continuity result}. The a priori estimate (%
\ref{bound on the second derivatives}) is a $W_{\mathrm{loc}}^{2,2}\left(
\Omega \right) -$bound of the local minimizer $u$ of the energy integral (%
\ref{energy integral}) only on subsets $\Omega ^{\prime }\subset \subset
\Omega $ of the type $\left\{ x\in \Omega ^{\prime }:\;g_{1}\left(
\left\vert Du\left( x\right) \right\vert \right) \geq m\right\} $ for fixed
constants $m>0$. For instance, when $g_{1}:\left[ 0,+\infty \right)
\rightarrow \left( 0,+\infty \right) $ is a positive function also at $t=0$,
i.e. when $m=g_{1}\left( 0\right) >0$ as it happens in the nondegenerate $p-$%
Laplacian and $p\left( x\right) -$Laplacian cases. In fact in these cases by
(\ref{bound on the second derivatives})$_{2}$ we get%
\begin{equation}
\text{the nondegenerate case:\ \ }\int_{B_{\rho }}\left\vert
D^{2}u\right\vert ^{2}\,dx\leq \tfrac{c}{m\left( R-\rho \right) ^{\theta
_{4}}}\left( \int_{B_{R}}\left( 1+f\left( Du\right) \right) \,dx\right)
^{\theta _{3}}.  \label{bound on the second derivatives (non degenerate)}
\end{equation}%
As an example related to the estimate (\ref{bound on the second derivatives}%
), we mention here the degenerate case recently studied by
Brasco-Carlier-Santambrogio \cite{Brasco-Carlier-Santambrogio 2010},
Santambrogio-Vespri \cite{Santambrogio-Vespri 2010}, Colombo-Figalli \cite%
{Colombo-Figalli 2014}, B\"{o}gelein-Duzaar-Giova-Passarelli-Scheven \cite%
{Boegelein-Duzaar-Giova-Passarelli-Scheven 2023}. It is a \textit{very
degenerate case}, with $f\left( x,t\right) =\frac{1}{p}\left( t-1\right)
_{+}^{p}-h\left( x\right) u$ (as before $t=\left\vert \xi \right\vert $) and
thus $g_{1}\left( t\right) =g_{2}\left( t\right) =0$ for all $t\in \left[ 0,1%
\right] $ and $g_{1}\left( t\right) >0$ when $t>1$. In this case, when $h=0,$
by Theorem \ref{Lipschitz continuity result} we deduce the $L^{\infty }$%
-local gradient bound (\ref{gradient bound}), while the $W_{\mathrm{loc}%
}^{2,2}\left( \Omega \right) -$bound (\ref{bound on the second derivatives})
gives contribution only at the subset of $\Omega $ where $\left\vert
Du\left( x\right) \right\vert >1$, being indefinite the gradient $Du\left(
x\right) $ and the matrix of its second derivatives $D^{2}u\left( x\right) $
when $\left\vert Du\left( x\right) \right\vert \leq 1$, since there any
Sobolev function $u$, with $\left\vert Du\left( x\right) \right\vert \leq 1$%
, a-priori could be a minimizer.

\section{Consequences and examples\label{Section: Consequences and examples}}

We collect some examples which satisfy the hypotheses (\ref{ellipticity})-(%
\ref{condition on beta}) and we state a specific version of Theorem \ref%
{Lipschitz continuity result} for each energy integral discussed below. For
the readers' convenience we start by reporting two well-known properties
which will be used to show that the models proposed below satisfy the
assumptions stated in the previous section. In Proposition \ref{quadratic
form} we represent the quadratic form, associated to the $n\times n$ matrix $%
(f_{\xi _{i}\xi _{j}})$ of the second derivatives of $f\left( x,\xi \right) $
with respect to $\xi $, in the particular case when $f$ depends on $\xi $
only thought its modulus $\left\vert \xi \right\vert $. We recall that $%
g_{1},g_{2},g_{3}:\left[ 0,+\infty \right) \rightarrow \left[ 0,+\infty
\right) $ are nonnegative increasing functions. They are not identically
equal to zero; the reason is to avoid the trivial case when the integrand $%
f=f\left( x,\xi \right) $ satisfies all conditions (\ref{ellipticity})-(\ref%
{condition on beta}) with $g_{1},g_{2},g_{3}$ identically equal to zero; in
this trivial case any $u\in W_{\mathrm{loc}}^{1,p}\left( \Omega \right) $
comes out to be a local minimizer of the "zero integral" and of course
regularity results do not hold. Thus there exists $t_{0}>0$ such that $%
g_{2}(t_{0})\geq g_{1}(t_{0})>0$ and, up to a rescaling, we can consider $%
t_{0}\leq 1$; therefore $g_{2}(1)\geq g_{1}(1)\geq c>0$. If necessary we
multiply $f$, and thus $g_{1},g_{2},g_{3}$ too, by the positive constant $%
1/c $ to have 
\begin{equation}
g_{2}(1)\geq g_{1}(1)\geq 1\,.  \label{g(1)}
\end{equation}%
We introduce a function $g:\Omega \times \left[ 0,+\infty \right)
\rightarrow \mathbb{R}$, $g=g\left( x,t\right) $, $t\in \left[ 0,+\infty
\right) $, $t=\left\vert \xi \right\vert $, $g$ twice differentiable with
respect to $t>0$, such that 
\begin{equation}
f\left( x,\xi \right) =g\left( x,\left\vert \xi \right\vert \right)
\,,\;\;\;\;\;\forall \;\left( x,\xi \right) \in \Omega \times \mathbb{R}%
^{n}\,.  \label{modulo dependence}
\end{equation}%
Although we explicitly remember that our Theorem \ref{Lipschitz continuity
result} deals with a general energy integrand $f\left( x,\xi \right) $ not
necessarily of the form (\ref{modulo dependence}).

\begin{proposition}
\label{quadratic form}Let $f\left( x,\xi \right) $ be represented in the
form $g\left( x,\left\vert \xi \right\vert \right) $, as in (\ref{modulo
dependence}), with partial derivative $g_{t}\left( x,t\right) $ locally
Lipschitz continuous in $t\in \left[ 0,+\infty \right) $ for fixed $x\in
\Omega $, with $g_{t}\left( x,0\right) =0$. Then the quadratic form
associated to the $n\times n$ matrix $(f_{\xi _{i}\xi _{j}})$ satisfies the
bounds 
\begin{equation}
g_{1}\left( x,\left\vert \xi \right\vert \right) \,\left\vert \lambda
\right\vert ^{2}\leq \sum_{i,j}\,f_{\xi _{i}\xi _{j}}\left( x,\xi \right)
\lambda _{i}\lambda _{j}\leq g_{2}\left( x,\left\vert \xi \right\vert
\right) \,\left\vert \lambda \right\vert ^{2},  \label{a0}
\end{equation}%
for all $\left( x,\xi \right) \in \Omega \times \mathbb{R}^{n}$, with $g_{1}$
and $g_{2}$ given by

(i) in general $g_{1}=\min \left\{ \frac{g_{t}}{t},g_{tt}\right\} $ and $%
g_{2}=\max \left\{ \frac{g_{t}}{t},g_{tt}\right\} $;

(ii) $g_{1}=\frac{g_{t}}{t}$ and $g_{2}=g_{tt}$, if $\frac{g_{t}\left(
x,t\right) }{t}$ is increasing with respect to $t$ at $x\in \Omega $;

(iii) $g_{1}=g_{tt}$ and $g_{2}=\frac{g_{t}}{t}$, if $\frac{g_{t}\left(
x,t\right) }{t}$ is decreasing with respect to $t$ at $x\in \Omega $.
\end{proposition}

Usually, it is not difficult to compute explicitly $g_{1}\left( \left\vert
\xi \right\vert \right) $ and $g_{2}\left( \left\vert \xi \right\vert
\right) $ and to check the monotoniticy of $g_{t}\left( x,t\right) /t$. For
instance for an integrand related to the $p-$Laplacian, of the type %
\eqref{modulo dependence}, with $f\left( x,\xi \right) =g\left( x,\left\vert
\xi \right\vert \right) =a\left( x\right) \left\vert \xi \right\vert ^{p}$
and $g\left( x,t\right) =a\left( x\right) t^{p}$, we have $g_{t}\left(
x,t\right) /t=pa\left( x\right) t^{p-2}$; therefore we are in the case 
\textit{(ii)} if $p\geq 2$, while \textit{(iii)} holds when $p\in \left( 1,2%
\right] $. In this case, if $c_{1}\leq a\left( x\right) \leq c_{2}$ for all $%
x\in \Omega $ (by $c_{i}$ we denote positive constants), then $%
c_{3}g_{t}/t\leq g_{tt}\leq c_{4}g_{t}/t$ for all $\left( x,t\right) \in
\Omega \times \left( 0,+\infty \right) $; or equivalently $c_{3}\leq
tg_{tt}/g_{t}\leq c_{4}$ for all $\left( x,t\right) \in \Omega \times \left(
0,+\infty \right) $. Please, note the "abuse of notation", due to the fact
that in Proposition \ref{quadratic form} $g_{1}$ and $g_{2}$ are functions
depending on $\left( x,t\right) =\left( x,\left\vert \xi \right\vert \right) 
$, while in the ellipticity condition (\ref{ellipticity}) they are functions 
$g_{1}=g_{1}\left( \left\vert \xi \right\vert \right) $, $g_{2}=g_{2}\left(
\left\vert \xi \right\vert \right) $ independent of $x$. Of course to go
from $g_{1}$ in Proposition \ref{quadratic form} to the $g_{1}$ in the left
hand side of (\ref{ellipticity}) we will consider the infimum of $%
g_{1}\left( x,\left\vert \xi \right\vert \right) $ with respect to $x$;
similarly, the supremum with respect to $x$ for $g_{2}$. In general, if $%
g_{1}\leq g_{2}\leq c_{5}g_{1}$ for all $\left( x,\xi \right) \in \Omega
\times \mathbb{R}^{n}$, we say that (\ref{a0}) are \textit{uniformly
elliptic conditions}. Therefore the $p-$Laplacian is an uniformly elliptic
operator. However in our Theorem \ref{Lipschitz continuity result} we do not
assume uniform elliptic conditions.

\bigskip

\begin{proof}[Proof of Proposition \protect\ref{quadratic form}]
For $f\left( x,\xi \right) =g\left( x,\left\vert \xi \right\vert \right) $,
as in (\ref{modulo dependence}), a computation shows 
\begin{equation*}
f_{\xi _{i}}=g_{t}\tfrac{\xi _{i}}{\left\vert \xi \right\vert }%
\,,\;\;\;\;\;f_{\xi _{i}\xi _{j}}=g_{tt}\tfrac{\xi _{i}\xi _{j}}{\left\vert
\xi \right\vert ^{2}}+g_{t}\tfrac{\delta _{ij}\left\vert \xi \right\vert
-\xi _{i}\frac{\xi _{j}}{\left\vert \xi \right\vert }}{\left\vert \xi
\right\vert ^{2}}=\xi _{i}\xi _{j}\left( \tfrac{g_{tt}}{\left\vert \xi
\right\vert ^{2}}-\tfrac{g_{t}}{\left\vert \xi \right\vert ^{3}}\right)
+\delta _{ij}\tfrac{g_{t}}{\left\vert \xi \right\vert },
\end{equation*}%
for $i,j=1,2,\ldots ,n$, where, as well known, $\delta _{ij}$ is the \textit{%
Kronecker delta}: $\delta _{ij}=1$ if $i=j$ and $\delta _{ij}=0$ if $i\neq j$%
. Therefore%
\begin{equation*}
\sum_{i,j=1}^{n}\,f_{\xi _{i}\xi _{j}}\left( x,\xi \right) \lambda
_{i}\lambda _{j}=\tfrac{1}{\left\vert \xi \right\vert ^{2}}\left( g_{tt}-%
\tfrac{g_{t}}{\left\vert \xi \right\vert }\right) \sum_{i,j=1}^{n}\,\xi
_{i}\xi _{j}\lambda _{i}\lambda _{j}+\tfrac{g_{t}}{\left\vert \xi
\right\vert }\left( \lambda _{1}^{2}+\lambda _{2}^{2}+\ldots \lambda
_{n}^{2}\right) .
\end{equation*}%
For the first sum we have $\sum_{i,j=1}^{n}\,\xi _{i}\xi _{j}\lambda
_{i}\lambda _{j}=\left( \sum_{i=1}^{n}\,\xi _{i}\lambda _{i}\right)
^{2}=\left( \xi ,\lambda \right) ^{2}$, where $\left( \xi ,\lambda \right) $
is the scalar product of $\xi ,\lambda \in \mathbb{R}^{n}$. Then we get $%
\sum_{i,j=1}^{n}\,f_{\xi _{i}\xi _{j}}\left( x,\xi \right) \lambda
_{i}\lambda _{j}=\frac{1}{\left\vert \xi \right\vert ^{2}}\left( g_{tt}-%
\frac{g_{t}}{\left\vert \xi \right\vert }\right) \left( \xi ,\lambda \right)
^{2}+\frac{g_{t}}{\left\vert \xi \right\vert }\left\vert \lambda \right\vert
^{2}$. By the Cauchy-Schwarz inequality $\left\vert \left( \xi ,\lambda
\right) \right\vert \leq \left\vert \xi \right\vert \,\left\vert \lambda
\right\vert $, for all $\left( x,t\right) \in \Omega \times \left[ 0,+\infty
\right) $ such that $g_{tt}-\frac{g_{t}}{\left\vert \xi \right\vert }\geq 0$
we finally deduce 
\begin{equation*}
\tfrac{g_{t}}{\left\vert \xi \right\vert }\left\vert \lambda \right\vert
^{2}\leq \sum_{i,j=1}^{n}\,f_{\xi _{i}\xi _{j}}\left( x,\xi \right) \lambda
_{i}\lambda _{j}\leq \tfrac{1}{\left\vert \xi \right\vert ^{2}}\left( g_{tt}-%
\tfrac{g_{t}}{\left\vert \xi \right\vert }\right) \left\vert \xi \right\vert
^{2}\,\left\vert \lambda \right\vert ^{2}+\tfrac{g_{t}}{\left\vert \xi
\right\vert }\left\vert \lambda \right\vert ^{2}=g_{tt}\left\vert \lambda
\right\vert ^{2},
\end{equation*}%
for all $\lambda \in \mathbb{R}^{n}$; while, if at $\left( x,t\right) \in
\Omega \times \left[ 0,+\infty \right) $ we have $g_{tt}-\frac{g_{t}}{%
\left\vert \xi \right\vert }\leq 0$ then 
\begin{equation*}
g_{tt}\left\vert \lambda \right\vert ^{2}=\tfrac{1}{\left\vert \xi
\right\vert ^{2}}\left( g_{tt}-\tfrac{g_{t}}{\left\vert \xi \right\vert }%
\right) \left\vert \xi \right\vert ^{2}\,\left\vert \lambda \right\vert ^{2}+%
\tfrac{g_{t}}{\left\vert \xi \right\vert }\left\vert \lambda \right\vert
^{2}\leq \sum_{i,j=1}^{n}\,f_{\xi _{i}\xi _{j}}\left( x,\xi \right) \lambda
_{i}\lambda _{j}\leq \tfrac{g_{t}}{\left\vert \xi \right\vert }\left\vert
\lambda \right\vert ^{2}\,,
\end{equation*}%
for all $\lambda \in \mathbb{R}^{n}$, which is equivalent to the conclusion
in \textit{(i)}.

By computing the partial derivative $\frac{\partial }{\partial t}\frac{%
g_{t}\left( x,t\right) }{t}=\frac{g_{tt}\left( x,t\right) t-g_{t}\left(
x,t\right) }{t^{2}}$ we see that at every $x\in \Omega $ such that $%
g_{tt}\left( x,t\right) t-g_{t}\left( x,t\right) \geq 0$ for all $t\in
\left( 0,+\infty \right) $, then $t\rightarrow \frac{g_{t}\left( x,t\right) 
}{t}$ is increasing in $\left( 0,+\infty \right) $ and vice versa. This
proves \textit{(ii)}. Similarly for \textit{(iii)} when $g_{tt}\left(
x,t\right) t-g_{t}\left( x,t\right) \leq 0$ at $x\in \Omega $, for all $t\in
\left( 0,+\infty \right) $.
\end{proof}

\bigskip

The following property is elementary and does not need a proof.

\begin{proposition}
\label{auxiliary property} Let $t_0 \in \mathbb{R}$. If $h(t)$ is a
continuous function in $[t_0,+\infty )$ such that $\lim_{t\rightarrow
+\infty }h(t)\in \mathbb{R}$, then $h(t)$ is bounded in $[t_0,+\infty )$.
\end{proposition}

\subsection{Anisotropic energy integrals\label{Section: Anisotropic energy
integrals}}

We emphasize that in this example the condition in (\ref{modulo dependence}%
), i.e. $f\left( x,\xi \right) =g\left( x,\left\vert \xi \right\vert \right) 
$, is not assumed. Not even $f\left( x,\xi \right) $ should depend on $\xi
=\left( \xi _{1},\xi _{2},\ldots ,\xi _{n}\right) \in \mathbb{R}^{n}$ only
through powers of its components $\left\vert \xi _{1}\right\vert ^{p_{1}}$, $%
\left\vert \xi _{2}\right\vert ^{p_{2}}$, $\ldots $, $\left\vert \xi
_{n}\right\vert ^{p_{n}}$, for instance as frequently represented in the
recent literature on this subject, in the form $f\left( x,\xi \right)
=\sum\limits_{i=1}^{n}a_{i}\left( x\right) \left\vert u_{x_{i}}\right\vert
^{p_{i}}$ for some exponents $p_{i}$, $i=1,2,\ldots ,n$. .

We fix some indices in the set $\left\{ 1,2,\ldots ,n\right\} $ in order to
form a not empty proper subset; for instance we fix the index $\left\{
n\right\} $ and we consider the energy integral 
\begin{equation}
\int_{\Omega }\,(\,\sum\limits_{i,j=1}^{n}a_{ij}\left( x\right)
u_{x_{i}}u_{x_{j}}+\left\vert u_{x_{n}}\right\vert ^{q}\,)\,dx\,,
\label{Anisotropic integral with p=2}
\end{equation}%
with $q\geq 2$, where $\left( a_{ij}\left( x\right) \right) $ is an $n\times
n$ positive definite matrix of locally Lipschitz continuous functions on $%
\Omega $. Precisely, for every set $\Omega ^{\prime }$ compactly contained
in $\Omega $ there exist positive constants $c_{1}$, $c_{2}$ such that $%
c_{1}\left\vert \lambda \right\vert ^{2}\leq
\sum\limits_{i,j=1}^{n}a_{ij}\left( x\right) \lambda _{i}\lambda _{j}\leq
c_{2}\left\vert \lambda \right\vert ^{2}$ for all $\lambda \in \mathbb{R}^{n}
$ and $x\in \Omega ^{\prime }$. Or more generally the anisotropic energy
integrals has the form 
\begin{equation}
\int_{\Omega }\left\{ h\left( x,Du\right) +\left\vert u_{x_{n}}\right\vert
^{q}\right\} \,dx\,,  \label{Anisotropic integral}
\end{equation}%
where $h$ is a real function defined in $\Omega \times \mathbb{R}^{n}$ whose
gradient $D_{\xi }h\left( x,\xi \right) $ is locally Lipschitz continuous in 
$\Omega \times \mathbb{R}^{n}$ with respect to $\xi \in \mathbb{R}^{n}$ and $%
q\geq p\geq 2$. Moreover $h\left( x,Du\right) $ is locally uniformly
elliptic and satisfies the standard (classical) $p-$ellipticity and growth
conditions 
\begin{equation}
c_{1}\left\vert \xi \right\vert ^{p-2}\left\vert \lambda \right\vert
^{2}\leq \sum_{i,j}\,h_{\xi _{i}\xi _{j}}\left( x,\xi \right) \lambda
_{i}\lambda _{j}\leq c_{2}\left( 1+\left\vert \xi \right\vert ^{p-2}\right)
\left\vert \lambda \right\vert ^{2},  \label{standard ellipticity for h}
\end{equation}%
\begin{equation}
\sum_{i}\,\left\vert h_{\xi _{i}x_{k}}\left( x,\xi \right) \right\vert \leq
c_{3}\left\vert \xi \right\vert ^{p-1},  \label{standard growth for h}
\end{equation}%
for all $\xi ,\lambda \in \mathbb{R}^{n}$ and $k=1,2,\ldots ,n$, a.e.\ $x\in
\Omega ^{\prime }$. In the particular case (\ref{Anisotropic integral with
p=2}), with $h\left( x,\xi \right) =\sum\limits_{i,j}a_{ij}\left( x\right)
\xi _{i}\xi _{j}$ and $p=2$, then $\sum\limits_{i,j}\,h_{\xi _{i}\xi
_{j}}\left( x,\xi \right) \lambda _{i}\lambda
_{j}=2\sum\limits_{i,j}a_{ij}\left( x\right) \lambda _{i}\lambda _{j}$ and (%
\ref{standard ellipticity for h}) is satisfied with $p=2$ (and different
constants). It is not difficult to test the ellipticity and growth
conditions of the integrand $f\left( x,\xi \right) =h\left( x,\xi \right)
+\left\vert \xi _{n}\right\vert ^{q}$; since $f_{\xi _{i}}=h_{\xi _{i}}$ for
all $i\in \left\{ 1,2,\ldots ,n-1\right\} $, $f_{\xi _{n}}=h_{\xi
_{n}}+q\left\vert \xi _{n}\right\vert ^{q-2}\xi _{n}$ , and $f_{\xi
_{i}x_{k}}=h_{\xi _{i}x_{k}}$ for all $i\in \left\{ 1,2,\ldots
,n-1,n\right\} $, then 
\begin{equation}
c_{1}\left\vert \xi \right\vert ^{p-2}\left\vert \lambda \right\vert
^{2}\leq \sum_{i,j}\,f_{\xi _{i}\xi _{j}}\left( x,\xi \right) \lambda
_{i}\lambda _{j}\leq c_{4}\left( 1+\left\vert \xi \right\vert ^{q-2}\right)
\left\vert \lambda \right\vert ^{2},  \label{standard ellipticity for f}
\end{equation}%
\begin{equation}
\sum_{i}\,\left\vert f_{\xi _{i}x_{k}}\left( x,\xi \right) \right\vert \leq
c_{3}\left\vert \xi \right\vert ^{p-1}.  \label{standard growth for f}
\end{equation}%
In order to test conditions (\ref{ellipticity})-(\ref{condition on beta}),
in this case we fix (up to the multiplicative constants) $g_{1}\left(
t\right) =t^{p-2}$, $g_{2}\left( t\right) =1+t^{q-2}$, $g_{3}\left( t\right)
=t^{p-1}$, under the usual notation $\left\vert \xi \right\vert =t$. We
start from (\ref{Assumption 3}), which takes the form 
\begin{equation*}
h\left( t\right) =\tfrac{t^{p-1}}{\left( 1+t^{\gamma }\right) \left(
t^{p-2}\right) ^{\frac{1}{2}}\left( 1+t^{q-2}\right) ^{\gamma -\frac{1}{2}}}=%
\tfrac{t^{p/2}}{\left( 1+t^{\gamma }\right) \left( 1+t^{q-2}\right) ^{\gamma
-\frac{1}{2}}}\leq M\,,
\end{equation*}%
for all $t\in \left[ 0,+\infty \right) $ and for a constant $M>0$. As in
Proposition \ref{auxiliary property}, $h\left( t\right) $ is a bounded
function in $\left[ 0,+\infty \right) $ if it has a finite limit as $%
t\rightarrow +\infty .$ This finite limit condition $\lim_{t\rightarrow
+\infty }h(t)\in \mathbb{R}$ is reduced to the exponents' inequality $\tfrac{%
p}{2}\leq \gamma +\left( q-2\right) \left( \gamma -\tfrac{1}{2}\right) \,,$
which, for $\gamma =1$, simply means $p\leq q$. Therefore assumption (\ref%
{Assumption 3}) in this case holds with $\gamma =1$. Then (\ref{11M})
corresponds to 
\begin{equation}
(1+t^{q-2})^{2\gamma -1}\,t^{2}\leq M\,(1+\int_{0}^{t}s^{\frac{p}{2}%
-1}ds)^{\alpha }=M\,(1+\tfrac{2}{p}t^{\frac{p}{2}})^{\alpha },  \label{a1}
\end{equation}%
which is satisfied by a positive constant $M$ and for all $t\geq 0$, if we
choose $\gamma =1$ and $\alpha $ in such a way that $\left[ \left(
q-2\right) \left( 2\gamma -1\right) \right] _{\gamma =1}+2=q\leq \alpha \,%
\frac{p}{2}$; i.e., if we choose $\alpha =2\frac{q}{p}$. To respect the
conditions $2\leq \alpha <2^{\ast }=:\frac{2n}{n-2}$ (there is not an upper
bound for $\alpha $ when $n=2$) we need 
\begin{equation*}
2\leq 2\tfrac{q}{p}<\tfrac{2n}{n-2}\;\;\;\Leftrightarrow \;\;\;1\leq \tfrac{q%
}{p}<\tfrac{n}{n-2}\,.
\end{equation*}%
Let us now discuss assumption (\ref{12M}). For $f\left( x,\xi \right) $ as
in (\ref{Anisotropic integral}), being $\left\vert f\left( x,\xi \right)
\right\vert \leq c_{5}\left( 1+\left\vert \xi \right\vert ^{q}\right) $, it
corresponds to $\left( 1+\left\vert \xi \right\vert ^{q-2}\right) ^{2\gamma
-1}\,\left\vert \xi \right\vert ^{2\gamma }\leq M\left\{ 1+\left\vert \xi
\right\vert ^{q}\right\} ^{\beta }$, which, in terms of exponents, gives $%
\left[ \left( q-2\right) \left( 2\gamma -1\right) +2\right] _{\gamma
=1}=q\leq \beta q$. Thus $\frac{q}{p}\leq \beta $. To respect the conditions
for $\beta $ in (\ref{condition on beta}), i.e. $1\leq \beta <\left[ \frac{%
2(\alpha -2+2\gamma )}{n(\alpha -4+2\gamma )}\right] _{\gamma =1}=\frac{%
2\alpha }{n(\alpha -2)}$, we choose $\beta =\frac{\alpha }{2}$ and we assume
the condition $1\leq \frac{\alpha }{2}<\frac{2\alpha }{n(\alpha -2)}$, which
is equivalent to $2\leq \alpha <2+\frac{4}{n}$. Recalling that we already
choose $\alpha =2\frac{q}{p}$, we finally impose the bound on $\frac{q}{p}$ 
\begin{equation}
1\leq \tfrac{q}{p}<1+\tfrac{2}{n}\,.
\label{bound of q/p in the anisotropic case}
\end{equation}%
As a consequence of Theorem \ref{Lipschitz continuity result}, with
parameters $\alpha =2\frac{q}{p}$, $\beta =\frac{\alpha }{2}=\frac{q}{p}$, $%
\gamma =1$, we have proved the following result.

\begin{example}[Lipschitz continuity result for anisotropic energy
inte\-grals]
\label{anisotr result} Let $2\leq p\leq q$ satisfy (\ref{bound of q/p in the
anisotropic case}). Every local minimizer in the Sobolev class (\ref{a
priori smooth class for local minimizers}) of the energy integral (\ref%
{Anisotropic integral}), or of the integral in (\ref{Anisotropic integral
with p=2}) when $p=2$, is locally Lipschitz continuous in $\Omega $ and
satisfies the uniform gradient estimate and the estimates for the $n\times n$
matrix $D^{2}u$ of the second derivatives of $u$ stated in Theorem \ref%
{Lipschitz continuity result}.
\end{example}

\begin{remark}
Similarly to the energy integrals (\ref{Anisotropic integral with p=2}),(\ref%
{Anisotropic integral}) considered in this Section \ref{Section: Anisotropic
energy integrals}, it is possible to deal with 
\begin{equation*}
\int_{\Omega }(\,\sum\limits_{i,j=1}^{n}a_{ij}\left( x\right)
u_{x_{i}}u_{x_{j}}+a\left( x\right) \left\vert u_{x_{n}}\right\vert
^{q}\,)\,dx\,\;\;\;\text{or}\;\;\;\int_{\Omega }\left\{ h\left( x,Du\right)
+a\left( x\right) \left\vert u_{x_{n}}\right\vert ^{q}\right\} \,dx\,,
\end{equation*}%
with a locally Lipschitz continuous coefficient $a\left( x\right) \geq 0$ in 
$\Omega $. Then the proof should proceed similarly to the proof given in
this section, by changing the choice of $\gamma =1$ with $\gamma =1+\delta $%
, with $\delta >0$ as explicitly described in Remark \ref{alphabetagamma}.
If $\delta >0$ $($and not $\delta =0$ as in the above proof$)$, then the
bound on the ratio $\frac{q}{p}$ in (\ref{bound of q/p in the anisotropic
case}) becomes more strict.
\end{remark}

\subsection{Exponential growth}

We consider the energy integral 
\begin{equation}
\int_{\Omega }f\left( x,Du\right) \,dx=\int_{\Omega }e^{a(x)|Du|^{\tau
}}\,dx\,  \label{exp integral}
\end{equation}%
related to the \textit{exponential growth case} $f(x,\xi )=e^{a(x)|\xi
|^{\tau }}$, with $a(x)$ positive locally Lipschitz function and $\tau \geq
2 $. For simplicity, here we consider $\tau =2$. Under the notations of
Proposition \ref{quadratic form} we have $f(x,\xi )=g\left( x,\left\vert \xi
\right\vert \right) $ and 
\begin{equation*}
g\left( x,t\right) =e^{a\left( x\right) t^{2}},\;\;g_{t}\left( x,t\right)
=2tae^{at^{2}},\;\;g_{tt}\left( x,t\right)
=4t^{2}a^{2}e^{at^{2}}+2ae^{at^{2}}.
\end{equation*}%
Then $g_{t}\left( x,t\right) /t=2a\left( x\right) e^{a\left( x\right) t^{2}}$%
, being product of positive increasing functions, is increasing with respect
to $t\in \left[ 0,+\infty \right) $ for all $x\in \Omega $. By \textit{(ii)}
of Proposition \ref{quadratic form}, \eqref{a0} holds with $g_{1}\left(
x,t\right) =g_{t}\left( x,t\right) /t$ and $g_{2}\left( x,t\right)
=g_{tt}\left( x,t\right) $.

Let $\Omega ^{\prime }$ be an open set compactly contained in $\Omega $ and $%
B_{R}$ a generic ball in $\Omega ^{\prime }$. The coefficient belongs to $W_{%
\mathrm{loc}}^{1,\infty }\left( \Omega \right) $ and thus is Lipschitz
continuous in $\Omega ^{\prime }$ with a Lipschitz constant $L$ uniform for
all $B_{R}\subset \Omega ^{\prime }$. We define 
\begin{equation}
p=\min \{a\left( x\right) :x\in \overline{B_{R}}\}\quad \mbox{and}\quad
q=\max \{a\left( x\right) :x\in \overline{B_{R}}\}.  \label{definion_of_p_q}
\end{equation}%
For every $\theta >1$ we can choose the radius $R=R\left( \theta \right) $
of the ball $B_{R}$ small enough such that $q\leq \theta p$; in fact $%
q-p=a\left( x_{2}\right) -a\left( x_{1}\right) \leq L\left\vert
x_{2}-x_{1}\right\vert \leq 2R$ and thus $q\leq p+2R<\theta p$ if we choose,
for instance, $\theta =1+3R/p$. Therefore 
\begin{equation*}
g_{t}\left( x,t\right) /t=2a\left( x\right) e^{a\left( x\right) t^{2}}\geq
2pe^{pt^{2}},\;\;\;\forall \;\left( x,t\right) \in \overline{B_{R}}\times %
\left[ 0,+\infty \right) \,
\end{equation*}%
and similarly, being $q\leq \theta p$, $g_{tt}\left( x,t\right) \leq 4t^{2}%
\left[ \theta p\right] ^{2}e^{\theta pt^{2}}+2\theta pe^{\theta pt^{2}}$.
Then, by considering the infimum with respect to $x\in B_{R}$ of $%
g_{t}\left( x,t\right) /t$ and the supremum with respect to $x\in B_{R}$ of $%
g_{tt}\left( x,t\right) $, the ellipticity condition \eqref{ellipticity} is
satisfied in $B_{R}$ with 
\begin{equation}
g_{1}(t)=c_{1}e^{pt^{2}}\quad \mbox{and}\quad g_{2}(t)=c_{2}\left(
1+t^{2}\right) e^{qt^{2}},  \label{g1g2exp}
\end{equation}%
and $c_{1}=2p$, $c_{2}=\left( 4p\right) ^{2}$ (with $\theta >1$ also bounded
above by $2$). 
Firstly, we verify \eqref{11M}. Then for $t\geq t_{0}\geq 1$, 
\begin{equation*}
\int_{0}^{t}\sqrt{g_{1}(s)}\,ds\geq \int_{0}^{t_{0}}\sqrt{g_{1}(s)}%
\,ds+p\int_{t_{0}}^{t}e^{\frac{ps^{2}}{2}}\,s\,ds=c+p\,e^{\frac{p}{2}t^{2}}.
\end{equation*}%
We apply Proposition \ref{auxiliary property} with $h(t):=((e^{qt^{2}}\left(
1+t^{2}\right) )^{2\gamma -1}\,t^{2})\,/\,(1+e^{\frac{p}{2}t^{2}})^{\alpha }$
and we require 
\begin{equation}
q(2\gamma -1)\leq \alpha \tfrac{p}{2}  \label{preupperboundratioqp}
\end{equation}%
in order to get a finite limit of $h(t)$ as $t\rightarrow +\infty $. We use
here the possibility stated after the definition \eqref{definion_of_p_q} of $%
p,q$; precisely, $\frac{q}{p}=\theta $, with $\theta $ arbitrarily close to $%
1$. Then \eqref{preupperboundratioqp} becomes $\theta p(2\gamma -1)\leq
\alpha p/2$. As in Remark \ref{alphabetagamma} we pose $\gamma =1+\delta $,
therefore we get $\theta p\,2\delta \leq p\left( \frac{\alpha }{2}-\theta
\right) $. Note that $\theta p$ is a bounded quantity, since $1\leq \theta
p\leq 2p$. We can choose $\alpha >2$ and $\theta $ sufficiently close to $1$%
, so that the right hand side is positive. Therefore we can also fix $\delta 
$ sufficiently close to $0$ in order to obtain the validity of the
inequality in \eqref{preupperboundratioqp}. For these values of $\alpha $
and $\gamma $ assumption \eqref{11M} holds. To verify \eqref{12M}, it is
enough to prove that 
\begin{equation}
(e^{qt^{2}}\left( 1+t^{2}\right) )^{2\gamma -1}\,t^{2\gamma }\leq
M\,(1+e^{pt^{2}})^{\beta }.  \label{12Mexp}
\end{equation}%
Proceeding as in the previous lines, \eqref{12Mexp} holds if $\theta
(2\delta +1)\leq \beta $. The previous inequality is trivially true for $%
\beta >1$, $\delta $ sufficiently close to $0$ and $\theta $ sufficiently
close to $1$. Now, we calculate $g_{tx_{k}}\left( x,\xi \right)
=2a_{x_{k}}(x)te^{a(x)t^{2}}\left( 1+a(x)t^{2}\right) $. Since $a(x)$ is a
locally Lipschitz function, then there exists a constant $L>0$ such that $%
\left\vert a_{x_{k}}(x)\right\vert \leq L$, so hypothesis 
\eqref{growth conditions
A} is satisfied with 
\begin{equation}
g_{3}(t)=c_{3}e^{\theta pt^{2}}t\left( 1+t^{2}\right) ,  \label{g3exp}
\end{equation}%
where $c_{3}=(1+2p)2L$ (here we have used that $q\leq \theta p$ and $%
1<\theta <2$). To verify \eqref{Assumption 3}, we should prove that 
\begin{equation*}
e^{\theta pt^{2}}t\left( 1+t^{2}\right) \leq M(1+t^{\gamma })\,(e^{pt^{2}})^{%
\frac{1}{2}}\,(e^{\theta pt^{2}}\left( 1+t^{2}\right) )^{\gamma -\frac{1}{2}%
}.
\end{equation*}%
If $t\rightarrow 0^{+}$, the left hand side of the inequality goes to $0$,
while the right hand side converges to a positive number, so %
\eqref{Assumption 3} holds true. When $t\rightarrow +\infty $, %
\eqref{Assumption 3} is reduced to the exponents' inequality $\theta \left( 
\frac{1}{2}-\delta \right) \leq \frac{1}{2}$, which is true by choosing $%
\delta $ sufficiently close to $0$ and $\theta $ sufficiently close to $1$.

As a consequence, we get the following consequence of Theorem \ref{Lipschitz
continuity result}, with parameters $\beta= \frac{\alpha}{2}+\delta$ and $%
\gamma= 1+\delta$, where $2<\alpha <2^*$, $1<\beta < 1+\frac{2}{n}$ and $0<
\delta <\frac{4}{n(n-2)}$ (see Remark \ref{alphabetagamma}).

\begin{example}[Lipschitz continuity result under exponential growth]
\label{Lipschitz continuity result exp growth} Every smooth local minimizer
of the energy integral (\ref{exp integral}) is locally Lipschitz continuous
in $\Omega $ and satisfies the uniform gradient estimate (\ref{gradient
bound}) and the estimates (\ref{bound on the second derivatives}) for the $%
n\times n$ matrix $D^{2}u$ of the second derivatives of $u$.
\end{example}

\subsection{$p(x)-$Laplacian and logarithm $p(x)-$Laplacian}

We consider the energy integral 
\begin{equation}
\int_{\Omega }f\left( x,Du\right) \,dx=\int_{\Omega }\left\vert
Du\right\vert ^{p(x)}\,dx\,  \label{p lapl integralDeg}
\end{equation}%
related to the \textit{$p(x)-$Laplacian degenerate case} $f(x,\xi
)=\left\vert \xi \right\vert ^{p(x)}$, with variable exponent $p(x)\in W_{%
\mathrm{loc}}^{1,\infty }(\Omega )$, $p(x)\geq 2$. Under the notations of
Proposition \ref{quadratic form} we have $f(x,\xi )=g\left( x,\left\vert \xi
\right\vert \right) $ and 
\begin{equation*}
g\left( x,t\right) =t^{p(x)},\quad g_{t}\left( x,t\right)
=p(x)t^{p(x)-1},\quad g_{tt}\left( x,t\right) =p(x)\left( p(x)-1\right)
t^{p(x)-2}.
\end{equation*}%
Then $g_{t}\left( x,t\right) /t=p(x)t^{p(x)-2}$ is locally Lipschitz
continuous with respect to $t\in \left[ 0,+\infty \right) $ for all $x\in
\Omega $. By \textit{(i)} of Proposition \ref{quadratic form}, \eqref{a0}
holds with 
\begin{equation*}
g_{1}\left( x,t\right) =p(x)t^{p(x)-2}\quad \mbox{and}\quad g_{2}\left(
x,t\right) =p(x)\left( p(x)-1\right) t^{p(x)-2}.
\end{equation*}%
Let $\Omega ^{\prime }$ be an open set compactly contained in $\Omega $ and $%
B_{R}$ a generic ball in $\Omega ^{\prime }$. We define $p=\min \{p\left(
x\right) :x\in \overline{B_{R}}\}$, $q=\max \{p\left( x\right) :x\in 
\overline{B_{R}}\}$. Given $\theta >1$, we choose the radius $R$ of the ball 
$B_{R}$ small enough such that $q\leq \theta p$. Therefore, for every $x\in 
\overline{B_{R}}$ and $t>1$, $g_{t}\left( x,t\right) /t\geq pt^{p-2}$, and
analogously, $g_{tt}(x,t)\leq \theta p(\theta p-1)t^{\theta p-2}$. Then, the
ellipticity condition \eqref{ellipticity} is satisfied in $B_{R}$ with 
\begin{equation}
g_{1}(t)=c_{1}t^{p-2}\quad \mbox{and}\quad g_{2}(t)=c_{2}t^{\theta p-2},
\label{g1g2p(x)Deg}
\end{equation}%
and $c_{1}=p$, $c_{2}=2p(2p-1)$ (where we used that $\theta >1$ is also
bounded above by $2$). Similarly for $t\in \lbrack 0,1]$, indeed in this
case it is sufficient to interchange the supremum with the infimum and
viceversa. So that we obtain $g_{1}(t)=c_{1}t^{\theta p-2}$ and $%
g_{2}(t)=c_{2}t^{p-2}.$ We observe that hypothesis \eqref{11M} is
automatically satisfied when $0\leq t\leq 1$, since the left hand side is
bounded from above and the right hand side is bounded by a positive constant
from below. In order to verify \eqref{11M} for $t>1$, we need to prove that $%
\left( c_{2}t^{\theta p-2}\right) ^{2\gamma -1}t^{2}\leq M\left( 1+\tfrac{2%
\sqrt{c_{1}}}{p}t^{\frac{p}{2}\alpha }\right)$ which, by Proposition \ref%
{auxiliary property}, follows from the condition $\left( \theta p-2\right)
\left( 2\gamma -1\right) +2\leq p\frac{\alpha }{2}$. As in Remark \ref%
{alphabetagamma} we pose $\gamma =1+\delta $, which gives 
\begin{equation}
\left( \theta p-2\right) 2\delta \leq p\left( \tfrac{\alpha }{2}-\theta
\right) \,.  \label{p(x) 3Deg}
\end{equation}%
Note that $\left( \theta p-2\right) $ is a bounded quantity, since $0\leq
\theta p-2\leq 2\left( p-1\right) $. The right hand side is positive if we
fix $\alpha >2$ and $\theta $ sufficiently close to $1$. Therefore we can
also fix $\delta $ sufficiently close to $0$ in order to obtain the validity
of the inequality in \eqref{p(x) 3Deg}. For these values of $\alpha $ and $%
\gamma $ assumption \eqref{11M} holds.

\noindent In order to verify \eqref{12M}, it is enough to show that 
\begin{equation}
\left( c_{2}t^{\theta p-2}\right) ^{2\gamma -1}t^{2\gamma }\leq M\left\{
1+t^{p}\right\} ^{\beta }.  \label{c_2 t}
\end{equation}%
Using Proposition \ref{auxiliary property}, inequality \eqref{c_2 t} is true
by imposing the following condition $(\theta p-2)(2\gamma -1)+2\gamma \leq
p\beta $. As previously done, by Remark \ref{alphabetagamma} we pose $\gamma
=1+\delta $, obtaining 
\begin{equation}
\left( \theta p-1\right) 2\delta \leq p\left( \beta -\theta \right) \,.
\label{p(x) 4Deg}
\end{equation}%
Note that $\left( \theta p-1\right) $ is a bounded quantity, since $0\leq
\theta p-1\leq 2p-1$. The right hand side is positive if we fix $\beta >1$
and $\theta $ sufficiently close to $1$. Therefore we can also fix $\delta $
sufficiently close to $0$ in order to obtain the validity of the inequality
in \eqref{p(x) 4Deg}. For these values of $\beta $ and $\gamma $ assumption %
\eqref{12M} holds. Now we compute $g_{tx_{k}}(x,t)=p_{x_{k}}(x)t^{p(x)-1}%
\left( 1+p(x)\log t\right) $. We observe that, even if $\left\vert
p_{x_{k}}\left( x\right) \right\vert $ is locally bounded, because of the
logarithmic factor, we cannot bound $\left\vert g_{tx_{k}}(x,t)\right\vert $
only with the exponential term $t^{p(x)-1}$. However, since $\log
t<t^{\omega }/\omega $ for every $\omega >0$, we can find a constant $c_{3}$%
, which depends on $\omega $ and on the $L^{\infty }\left( B_{R}\right) $%
-bounds of $p(x)$ and $p_{x_{k}}\left( x\right) $, such that $\left\vert
g_{tx_{k}}(x,t)\right\vert \leq c_{3}\left( 1+t^{\theta p-1+\omega }\right)
. $ So we fix 
\begin{equation}
g_{3}(t)=c_{3}\left( 1+t^{\theta p-1+\omega }\right) \,.  \label{g3plaplDeg}
\end{equation}%
In order to verify \eqref{Assumption 3}, we need to find a constant $M$ such
that 
\begin{equation}
c_{3}\left( 1+t^{\theta p-1+\omega }\right) \,\leq M(1+t^{\gamma })\,\left(
c_{1}t^{p-2}\right) ^{\frac{1}{2}}\,\left( c_{2}t^{\theta p-2}\right)
^{\gamma -\frac{1}{2}}.  \label{assump3p(x)lapl}
\end{equation}%
We apply Proposition \ref{auxiliary property}, imposing the following
inequality on the exponents 
\begin{equation}
\theta p\left( \tfrac{3}{2}-\gamma \right) +\omega \leq \tfrac{p}{2}%
+1-\gamma \,.  \label{p(x) a}
\end{equation}%
We can verify that $\gamma =1$ is not a possible choice, since $\tfrac{1}{2}%
q+\omega \leq \tfrac{1}{2}\theta p+\omega \leq \tfrac{1}{2}p$ is a false
inequality even if $q=p$. Then we follow Remark \ref{alphabetagamma} by
posing $\gamma =1+\delta $ (and $\beta =\frac{\alpha }{2}+\delta $); from (%
\ref{p(x) a}) we get $\theta p\left( \tfrac{1}{2}-\delta \right) +\omega
\leq \tfrac{p}{2}-\delta $, or equivalently $\tfrac{1}{2}\theta p+\omega
\leq \tfrac{p}{2}+\delta \left( q-1\right) $, which in terms of the
parameter $\delta $ means (we choose for $\delta $ the minimum possible
value; i.e., when the equality sign holds) 
\begin{equation}
\delta =\tfrac{\left( \theta -1\right) p+2\omega }{2\left( \theta p-1\right) 
}\,.  \label{p(x) b}
\end{equation}%
Finally we observe that the value of $\delta >0$ in (\ref{p(x) b}) can be
fixed as small as we like, by choosing $\theta $ close to $1$ and $\omega $
close to $0$. Therefore we can fix $\theta >1$ and $\omega >0$ such that the
condition $0\leq \delta <\frac{4}{n(n-2)}$ in Remark \ref{alphabetagamma} is
satisfied.

At this point, we can state the following consequence of Theorem \ref%
{Lipschitz continuity result}.

\begin{example}[Lipschitz continuity for $p(x)$ Laplacian integral]
\label{Lipschitz continuity result p(x) Laplacian Deg} There exist $\alpha
>2 $, $\beta >1$, $\gamma >1$ and functions $g_{1}$, $g_{2}$, $g_{3}$
defined as (\ref{g1g2p(x)Deg}) and (\ref{g3plaplDeg}) satisfying the
ellipticity and growth conditions (\ref{ellipticity})-(\ref{condition on
beta}), such that every smooth local minimizer of the energy integral (\ref%
{p lapl integralDeg}) is locally Lipschitz continuous in $\Omega $ and
satisfies the uniform gradient estimate (\ref{gradient bound}) and the
estimates (\ref{bound on the second derivatives}) for the $n\times n$ matrix 
$D^{2}u$ of the second derivatives of $u$.
\end{example}

Similar computations can be carried out for the \textit{Orlicz type energy
functionals $($logarithm $p(x)-$Laplacian$)$} 
\begin{equation*}
f\left( x,\xi \right) =\left\vert \xi \right\vert ^{p(x)}\log (1+\left\vert
\xi \right\vert ^{2}),
\end{equation*}%
where $p(x)\geq 2$ is a local Lipschitz continuous exponent. In this case we
have to test the hypotheses set under the choice 
\begin{equation*}
g_{1}\left( t\right) =t^{p-2}\log (1+t^{2}),\quad g_{2}\left( t\right)
=t^{\theta p-2}\left( \log (1+t^{2})+1\right) \quad g_{3}(t)=\left(
1+t^{\theta p-1+\omega }\right) ,
\end{equation*}%
where we have omitted multiplicative constants for simplicity.


\subsection{Double phase case\label{Degenerate double phase case}}

We consider the energy integral 
\begin{equation}
\int_{\Omega }f\left( x,Du\right) \,dx=\int_{\Omega }\left\{\left\vert
Du\right\vert ^{p}+a(x)\left\vert Du\right\vert ^{q}\right\} \,dx\,
\label{deg double phase integral}
\end{equation}%
related to the \textit{double phase degenerate case} $f\left( x,\xi \right)
=\left\vert \xi\right\vert ^{p}+a(x)\left\vert \xi\right\vert ^{q}, \label%
{deg original case}$ with $2\leq p\leq q$ and $a$ nonnegative locally
Lipschitz continuous function in $\Omega $. We fix 
\begin{equation*}
g_{1}\left( \left\vert \xi \right\vert \right) =|\xi |^{p-2,}\quad
g_{2}\left( \left\vert \xi \right\vert \right) =|\xi |^{q-2} \quad \mbox{and}%
\quad g_{3}\left( \left\vert \xi \right\vert \right) =|\xi |^{q-1},
\end{equation*}%
where we have omitted multiplicative constants for simplicity.

We introduce a further exponent, greater than $q$, at the same distance $q-p$
from $q$; precisely $q+\left( q-p\right) =2q-p$. Then for any fixed $b\geq 0$%
, we consider the auxiliary variational problem, related to a \textit{%
multiple phase}, of the form 
\begin{equation}
f\left( x,\xi \right) =\left\vert \xi \right\vert ^{p}+a\left( x\right)
\left\vert \xi \right\vert ^{q}+b\left\vert \xi \right\vert ^{2q-p}
\label{deg auxiliary case}
\end{equation}%
and $g_{1}\left( \left\vert \xi \right\vert \right) =\left\vert \xi
\right\vert ^{p-2}$, $g_{2}\left( \left\vert \xi \right\vert \right)
=\left\vert \xi \right\vert ^{2q-p-2}$, $g_{3}\left( \left\vert \xi
\right\vert \right) =\left\vert \xi \right\vert ^{q-1}$. We apply the
regularity Theorem \ref{Lipschitz continuity result}. In particular
assumption (\ref{Assumption 3}) trivially holds with $\gamma =1$. With the
aim to test assumption \eqref{11M}, we note that 
\begin{equation*}
\int_{0}^{t}\sqrt{g_{1}(s)}ds=\int_{0}^{t}s^{\frac{p-2}{2}}ds=\tfrac{2\,t^{%
\frac{p}{2}}}{p};
\end{equation*}%
therefore in order to obtain the validity of \eqref{11M} we must impose $%
2q-p-2+2=2q-p\leq \frac{p}{2}\alpha $, that is $\frac{q}{p}\leq 1+\frac{%
\alpha }{2}$. Taking into account the bound $2\leq \alpha <2^{\ast }\;%
\underset{\text{if }n>2}{=}\;\;\frac{2n}{n-2}$ for $\alpha $ in 
\eqref{condition on
alpha} and letting $\alpha $ converge to the limit value $2^{\ast }$ we
finally obtain 
\begin{equation}
\tfrac{q}{p}<1+\tfrac{2^{\ast }}{2}\;\underset{\text{if }n>2}{=}\;\;\tfrac{n%
}{n-2}=1+\tfrac{2}{n-2}\,.  \label{deg bound on q/p for the double phase}
\end{equation}%
Assumption \eqref{12M} requires that $t^{2q-p}\leq M\left\{
1+t^{p}+a(x)t^{q}+bt^{2q-p}\right\} ^{\beta }$, that is $\beta \geq 1$,
which is true by assumption \eqref{condition on beta}.

We remark that in the above considerations we used a nonnegative constant
coefficient $b$ which can be chosen equal to zero too, the case which
corresponds to the original energy integral \eqref{deg double phase integral}%
. Thus we have proved the following result, as a consequence of the general
Theorem \ref{Lipschitz continuity result}.

\begin{example}
Under the bound (\ref{deg bound on q/p for the double phase}) on the ratio $%
q/p$ every local minimizer $u\in W_{\mathrm{loc}}^{1,2q-p}\left( \Omega
\right) $ to the double phase energy integral in (\ref{deg double phase
integral}) (or, when $b\neq 0,$ a minimizer of the integral related to the
auxiliary integrand (\ref{deg auxiliary case}) as well) is locally Lipschitz
continuous in $\Omega $ and satisfies the uniform gradient estimate (\ref%
{gradient bound}) and the estimates (\ref{bound on the second derivatives})
for the $n\times n$ matrix $D^{2}u$ of the second derivatives of $u$.
\end{example}

\bigskip

\section{Proof of Theorem \protect\ref{Lipschitz continuity result}}

\label{Section of the proofs}

\subsection{Step 1. Second variation}

We consider a local minimizer $u$ of the energy integral (\ref{energy
integral}) under the supplementary assumptions 
\eqref{a priori smooth class
for local minimizers}. The local minimizer $u$ satisfies the Euler's first
variation in the weak form 
\begin{equation*}
\int_{\Omega }\,\sum_{i=1}^{n}f_{\xi _{i}}\left( x,Du\right) \psi
_{x_{i}}\,dx=0\,,
\end{equation*}%
for every $\psi \in W_{0}^{1,2}\left( \Omega \right) $. Since $u\in W_{%
\mathrm{loc}}^{2,2}\left( \Omega \right) $, we consider a generic integer $%
k\in \left\{ 1,2,\ldots ,n\right\} $ and $\psi =-\frac{\partial \varphi }{%
\partial x_{k}}=-\varphi _{x_{k}}$, where $\varphi \in W_{0}^{1,2}\left(
\Omega \right) $ is a generic test function. By integrating by parts we get 
\begin{equation}
\int_{\Omega }\left( \sum_{i,j=1}^{n}f_{\xi _{i}\xi _{j}}\left( x,Du\right)
u_{x_{j}x_{k}}+\sum_{i=1}^{n}f_{\xi _{i}x_{k}}\left( x,Du\right) \right)
\varphi _{x_{i}}\,dx=0\,.  \label{1 in Step 1}
\end{equation}%
We make a further choice of the test function $\varphi $, by posing $\varphi
=\eta ^{2}u_{x_{k}}\Phi \left( \left\vert Du\right\vert \right) $, where $%
\eta \in W_{0}^{1,2}\left( \Omega \right) $ and $\Phi :\left( 0,+\infty
\right) \rightarrow \left( 0,+\infty \right) $ is a nonnegative, increasing,
locally Lipschitz continuous function in $\left( 0,+\infty \right) $, to be
chosen later. Then 
\begin{equation*}
\varphi _{x_{i}}=2\eta \eta _{x_{i}}u_{x_{k}}\Phi \left( \left\vert
Du\right\vert \right) +\eta ^{2}u_{x_{k}x_{i}}\Phi \left( \left\vert
Du\right\vert \right) +\eta ^{2}u_{x_{k}}\Phi ^{\prime }\left( \left\vert
Du\right\vert \right) \left\vert Du\right\vert _{x_{i}}.
\end{equation*}

\subsubsection{Terms to estimate from \protect\fbox{\textbf{1}} to \protect%
\fbox{\textbf{6}}\label{Section: all the addenda}}

The left hand side in (\ref{1 in Step 1}) now splits into six addenda 
\begin{align}
& \fbox{\textbf{1}}\;\;\;\int_{\Omega }2\eta \Phi \left( \left\vert
Du\right\vert \right) \sum_{i,j=1}^{n}f_{\xi _{i}\xi _{j}}\left( x,Du\right)
u_{x_{j}x_{k}}\eta _{x_{i}}u_{x_{k}}\,dx  \notag  \label{2 in Step 1} \\
& \fbox{\textbf{2}}\;\;\;+\int_{\Omega }\eta ^{2}\Phi \left( \left\vert
Du\right\vert \right) \sum_{i,j=1}^{n}f_{\xi _{i}\xi _{j}}\left( x,Du\right)
u_{x_{j}x_{k}}u_{x_{k}x_{i}}\,dx  \notag \\
& \fbox{\textbf{3}}\;\;\;+\int_{\Omega }\eta ^{2}\Phi ^{\prime }\left(
\left\vert Du\right\vert \right) \sum_{i,j=1}^{n}f_{\xi _{i}\xi _{j}}\left(
x,Du\right) u_{x_{j}x_{k}}u_{x_{k}}\left\vert Du\right\vert _{x_{i}}\,dx 
\notag \\
& \fbox{\textbf{4}}\;\;\;+\int_{\Omega }2\eta \Phi \left( \left\vert
Du\right\vert \right) \sum_{i=1}^{n}f_{\xi _{i}x_{k}}\left( x,Du\right) \eta
_{x_{i}}u_{x_{k}}\,dx  \notag \\
& \fbox{\textbf{5}}\;\;\;+\int_{\Omega }\eta ^{2}\Phi \left( \left\vert
Du\right\vert \right) \sum_{i=1}^{n}f_{\xi _{i}x_{k}}\left( x,Du\right)
u_{x_{k}x_{i}}\,dx  \notag \\
& \fbox{\textbf{6}}\;\;\;+\int_{\Omega }\eta ^{2}\Phi ^{\prime }\left(
\left\vert Du\right\vert \right) \sum_{i=1}^{n}f_{\xi _{i}x_{k}}\left(
x,Du\right) u_{x_{k}}\left\vert Du\right\vert _{x_{i}}\,dx=0\,.
\end{align}%
In the following sections we estimate (\ref{2 in Step 1}) term by term.

\subsection{Step 2. Main estimates}

In the proof of Theorem \ref{Lipschitz continuity result} we use several
times the elementary inequality $\left( \sqrt{\varepsilon }a-\frac{1}{2\sqrt{%
\varepsilon }}b\right) ^{2}\geq 0$; i.e., $ab\leq \varepsilon a^{2}+\frac{1}{%
4\varepsilon }b^{2}$, valid for all $a,b\in \mathbb{R}$ and $\varepsilon >0$.

\subsubsection{Estimate of the term in $\protect\fbox{\textbf{1}}$}

In order to estimate of the term in $\fbox{\textbf{1}}$ we use here the 
\textit{Cauchy-Schwarz inequality} for symmetric quadratic forms 
\begin{align*}
& \left\vert \int_{\Omega }2\eta \Phi \left( \left\vert Du\right\vert
\right) \sum_{i,j=1}^{n}f_{\xi _{i}\xi _{j}}\left( x,Du\right)
u_{x_{j}x_{k}}\eta _{x_{i}}u_{x_{k}}\,dx\right\vert  \\
& \quad \leq \int_{\Omega }2\Phi \left( \left\vert Du\right\vert \right)
\left( \eta ^{2}\sum_{i,j=1}^{n}f_{\xi _{i}\xi _{j}}\left( x,Du\right)
u_{x_{i}x_{k}}u_{x_{j}x_{k}}\right) ^{1/2} \\
& \quad \quad \cdot \left( \left( u_{x_{k}}\right)
^{2}\sum_{i,j=1}^{n}f_{\xi _{i}\xi _{j}}\left( x,Du\right) \eta _{x_{i}}\eta
_{x_{j}}\right) ^{1/2}\,dx \\
& \quad \leq 2\varepsilon \int_{\Omega }\eta ^{2}\Phi \left( \left\vert
Du\right\vert \right) \sum_{i,j=1}^{n}f_{\xi _{i}\xi _{j}}\left( x,Du\right)
u_{x_{i}x_{k}}u_{x_{j}x_{k}}\,dx \\
& \quad \quad +\frac{1}{2\varepsilon }\int_{\Omega }\Phi \left( \left\vert
Du\right\vert \right) \left( u_{x_{k}}\right) ^{2}\sum_{i,j=1}^{n}f_{\xi
_{i}\xi _{j}}\left( x,Du\right) \eta _{x_{i}}\eta _{x_{j}}\,dx.
\end{align*}%
By the \textit{growth condition} in the right hand side of the (\ref%
{ellipticity}), for $\varepsilon =\frac{1}{4}$ we obtain%
\begin{equation*}
\left\vert \int_{\Omega }2\eta \Phi \left( \left\vert Du\right\vert \right)
\sum_{i,j=1}^{n}f_{\xi _{i}\xi _{j}}\left( x,Du\right) u_{x_{j}x_{k}}\eta
_{x_{i}}u_{x_{k}}\,dx\right\vert 
\end{equation*}%
\thinspace \thinspace \thinspace 
\begin{equation*}
\leq \tfrac{1}{2}\int_{\Omega }\eta ^{2}\Phi \left( \left\vert Du\right\vert
\right) \sum_{i,j=1}^{n}f_{\xi _{i}\xi _{j}}\left( x,Du\right)
u_{x_{i}x_{k}}u_{x_{j}x_{k}}\,dx
\end{equation*}%
\begin{equation}
+2\int_{\Omega }\Phi \left( \left\vert Du\right\vert \right) \left(
u_{x_{k}}\right) ^{2}g_{2}\left( \left\vert Du\right\vert \right)
\,\left\vert D\eta \right\vert ^{2}\,dx\,.  \label{Estimate of the term in 1}
\end{equation}%
We insert this inequality in the equation (\ref{2 in Step 1}) obtained in
the previous Section \ref{Section: all the addenda}. We chose to put some
addenda in the left hand side and some others in the opposite side; the
reason is clear by taking into account the sign of all addenda and/or their
absolute values. In particular $\fbox{\textbf{2}}$ and $\fbox{\textbf{3}}$
are nonnegative as consequence of the ellipticity condition in the left hand
side of (\ref{ellipticity}). Precisely $\fbox{\textbf{3}}$ is nonnegative
after a sum with respect to $k\in \left\{ 1,2,\ldots ,n\right\} $; see
Section \ref{Subsection: Estimate of the term in 3}. Since $\left\vert \,%
\fbox{\textbf{1}}\,\right\vert \leq \frac{1}{2}\,\fbox{\textbf{2}}$ plus the
last term in (\ref{Estimate of the term in 1}), we get 
\begin{equation}
\frac{1}{2}\,\fbox{\textbf{2}}+\fbox{\textbf{3}}\leq \left\vert \,\fbox{%
\textbf{4}}\,\right\vert +\left\vert \,\fbox{\textbf{5}}\,\right\vert
+\left\vert \,\fbox{\textbf{6}}\,\right\vert +2\int_{\Omega }\left\vert
D\eta \right\vert ^{2}\Phi \left( \left\vert Du\right\vert \right) \left(
u_{x_{k}}\right) ^{2}g_{2}\left( \left\vert Du\right\vert \right) \,dx\,.
\label{Estimate of the term in 1 bis}
\end{equation}%
We observe that, when we sum both sides over $k\in \left\{ 1,2,\ldots
,n\right\} $, in the right hand side of (\ref{Estimate of the term in 1 bis}%
) we have $\sum_{k=1}^{n}\left( u_{x_{k}}\right) ^{2}=\left\vert
Du\right\vert ^{2}$.

\subsubsection{Use of ellipticity in the term in $\protect\fbox{\textbf{2}}$}

This is the simplest step among these estimates. We bound from below the
addendum in $\fbox{\textbf{2}}$ by mean of the \textit{ellipticity condition}
in the left hand side of (\ref{ellipticity}) 
\begin{equation*}
\int_{\Omega }\eta ^{2}\Phi \left( \left\vert Du\right\vert \right)
\sum_{i,j=1}^{n}f_{\xi _{i}\xi _{j}}\left( x,Du\right)
u_{x_{j}x_{k}}u_{x_{k}x_{i}}\,dx
\end{equation*}%
\begin{equation*}
\geq \int_{\Omega }\eta ^{2}\Phi \left( \left\vert Du\right\vert \right)
g_{1}\left( \left\vert Du\right\vert \right) \,\left\vert
Du_{x_{k}}\right\vert ^{2}\,dx\,,
\end{equation*}%
by noting that for Sobolev functions $u_{x_{j}x_{k}}=u_{x_{k}x_{j}}$ a.e. in 
$\Omega $. When we sum both sides over $k=1,2,\ldots ,n$ we observe that $%
\sum_{k=1}^{n}\left\vert Du_{x_{k}}\right\vert ^{2}=\left\vert
D^{2}u\right\vert ^{2}$.

\subsubsection{Use of ellipticity in the term in $\protect\fbox{\textbf{3}}$%
\label{Subsection: Estimate of the term in 3}}

In order to estimate the term in $\fbox{\textbf{3}}$, we explicitly
represent $\left\vert Du\right\vert _{x_{i}}$, recalling that for Sobolev
functions the second order mixed derivatives coincide 
\begin{equation}
\left\vert Du\right\vert _{x_{i}}=\frac{\partial }{\partial x_{i}}%
(\sum_{k=1}^{n}\left( u_{x_{k}}\right) ^{2})^{1/2}=\frac{%
\sum_{k=1}^{n}u_{x_{k}}u_{x_{k}x_{i}}}{\left\vert Du\right\vert }\,.
\label{gradient of modulus of Du 1}
\end{equation}%
Thus, 
\begin{equation}
\frac{1}{\left\vert Du\right\vert }\sum_{i,j,k=1}^{n}f_{\xi _{i}\xi
_{j}}\left( x,Du\right) u_{x_{j}x_{k}}u_{x_{k}}(\left\vert Du\right\vert
)_{x_{i}}=\sum_{i,j=1}^{n}f_{\xi _{i}\xi _{j}}\left( x,Du\right) (\left\vert
Du\right\vert )_{x_{j}}(\left\vert Du\right\vert )_{x_{i}}.
\label{gradient of modulus of Du 2}
\end{equation}%
Under the notation $D\left( \left\vert Du\right\vert \right) =\left(
\left\vert Du\right\vert _{x_{1}},\left\vert Du\right\vert _{x_{2}},\ldots
,\left\vert Du\right\vert _{x_{n}}\right) $ and $\sum_{i=1}^{n}\left(
\left\vert Du\right\vert _{x_{i}}\right) ^{2}=\left\vert D\left( \left\vert
Du\right\vert \right) \right\vert ^{2}$, we sum over $k$ in $\fbox{\textbf{3}%
}$ and we get 
\begin{align}
& \int_{\Omega }\eta ^{2}\Phi ^{\prime }\left( \left\vert Du\right\vert
\right) \sum_{i,j,k=1}^{n}f_{\xi _{i}\xi _{j}}\left( x,Du\right)
u_{x_{j}x_{k}}u_{x_{k}}(\left\vert Du\right\vert )_{x_{i}}\,dx
\label{gradient of modulus of Du 3} \\
& \quad =\int_{\Omega }\eta ^{2}\Phi ^{\prime }\left( \left\vert
Du\right\vert \right) \left\vert Du\right\vert \sum_{i,j=1}^{n}f_{\xi
_{i}\xi _{j}}\left( x,Du\right) (\left\vert Du\right\vert
)_{x_{j}}(\left\vert Du\right\vert )_{x_{i}}\,dx  \notag \\
& \quad \geq \int_{\Omega }\eta ^{2}\Phi ^{\prime }\left( \left\vert
Du\right\vert \right) \left\vert Du\right\vert g_{1}\left( \left\vert
Du\right\vert \right) \,\left\vert D\left( \left\vert Du\right\vert \right)
\right\vert ^{2}\,dx,  \notag
\end{align}%
where in the last inequality we have used again the \textit{ellipticity
condition} in the left hand side of (\ref{ellipticity}).

\begin{remark}
We use in (\ref{gradient of modulus of Du 1}) the \textit{Cauchy-Schwarz
inequality} 
\begin{equation*}
\left\vert Du\right\vert _{x_{i}}=\frac{\sum_{k=1}^{n}u_{x_{k}}u_{x_{k}x_{i}}%
}{\left\vert Du\right\vert }\leq \frac{\left\vert Du\right\vert
(\sum_{k=1}^{n}\left( u_{x_{k}x_{i}}\right) ^{2})^{\frac{1}{2}}}{\left\vert
Du\right\vert }=(\sum_{k=1}^{n}\left( u_{x_{i}x_{k}}\right) ^{2})^{\frac{1}{2%
}}=\left\vert Du_{x_{i}}\right\vert .
\end{equation*}%
Therefore 
\begin{equation}
\left\vert D\left( \left\vert Du\right\vert \right) \right\vert
^{2}=\sum_{i=1}^{n}\left( \left\vert Du\right\vert _{x_{i}}\right) ^{2}\leq
\sum_{i=1}^{n}\left\vert Du_{x_{i}}\right\vert ^{2}
\label{gradient of modulus of Du 4}
\end{equation}%
\begin{equation*}
=\sum_{i=1}^{n}(\,\sum_{j=1}^{n}\left( u_{x_{i}x_{j}}\right)
^{2}\,)=\sum_{i,j=1}^{n}\left( u_{x_{i}x_{j}}\right) ^{2}=\left\vert
D^{2}u\right\vert ^{2}.
\end{equation*}%
For the term in $\fbox{\textbf{3}}$, as represented in (\ref{gradient of
modulus of Du 3}), by the growth condition in the right hand side of (\ref%
{ellipticity}) we also deduce the estimate 
\begin{equation*}
0\leq \int_{\Omega }\eta ^{2}\Phi ^{\prime }\left( \left\vert Du\right\vert
\right) \left\vert Du\right\vert \sum_{i,j=1}^{n}f_{\xi _{i}\xi _{j}}\left(
x,Du\right) (\left\vert Du\right\vert )_{x_{j}}(\left\vert Du\right\vert
)_{x_{i}}\,dx
\end{equation*}%
\begin{equation*}
\leq \int_{\Omega }\eta ^{2}\Phi ^{\prime }\left( \left\vert Du\right\vert
\right) \left\vert Du\right\vert g_{2}\left( \left\vert Du\right\vert
\right) \,\left\vert D\left( \left\vert Du\right\vert \right) \right\vert
^{2}\,dx
\end{equation*}%
\begin{equation}
\leq \int_{\Omega }\eta ^{2}\Phi ^{\prime }\left( \left\vert Du\right\vert
\right) \left\vert Du\right\vert g_{2}\left( \left\vert Du\right\vert
\right) \,\left\vert D^{2}u\right\vert ^{2}\,dx\,.
\label{gradient of modulus of Du 5}
\end{equation}%
At the beginning of Section \ref{Section: Step 5}, for the first step of the
iteration procedure, we fix $\Phi \left( t\right) =t$ for all $t\in \left[
0,+\infty \right) $, thus $\Phi ^{\prime }\left( t\right) =1$. Therefore, as
a consequence of (\ref{a priori smooth class for local minimizers}), (\ref%
{gradient of modulus of Du 5}) shows that the term in $\fbox{\textbf{3}}$ is
finite, as well as all the other addenda listed in Section \ref{Section: all
the addenda}, when the iteration procedure starts with $\Phi \left( t\right)
\equiv t$.
\end{remark}

\subsubsection{Estimate of the term in $\protect\fbox{\textbf{4}}$}

For any fixed $k\in \left\{ 1,2,\ldots ,n\right\} $, by the growth
conditions (\ref{growth conditions A}) and assumption (\ref{Assumption 3})
we obtain 
\begin{align*}
& \left\vert \int_{\Omega }2\eta \Phi \left( \left\vert Du\right\vert
\right) \sum_{i=1}^{n}f_{\xi _{i}x_{k}}\left( x,Du\right) \eta
_{x_{i}}u_{x_{k}}\,dx\right\vert \\
& \quad \leq \int_{\Omega }2\eta \left\vert D\eta \right\vert \Phi \left(
\left\vert Du\right\vert \right) \left\vert Du\right\vert
\sum_{i=1}^{n}\left\vert f_{\xi _{i}x_{k}}\left( x,Du\right) \right\vert \,dx
\\
& \quad \leq \int_{\Omega }2\eta \left\vert D\eta \right\vert \Phi \left(
\left\vert Du\right\vert \right) \left\vert Du\right\vert g_{3}\left(
\left\vert Du\right\vert \right) \,dx \\
& \quad \leq 2M\int_{\Omega }\eta \left\vert D\eta \right\vert \Phi \left(
\left\vert Du\right\vert \right) \left\vert Du\right\vert \left(
1+\left\vert Du\right\vert ^{\gamma }\right) \,\left( g_{1}\left( \left\vert
Du\right\vert \right) \right) ^{\frac{1}{2}}\,\left( g_{2}\left( \left\vert
Du\right\vert \right) \right) ^{\gamma -\frac{1}{2}}\,dx \\
& \quad \leq 2M\int_{\Omega }\eta \left\vert D\eta \right\vert \Phi \left(
\left\vert Du\right\vert \right) \left\vert Du\right\vert \left(
1+\left\vert Du\right\vert ^{\gamma }\right) \left( g_{2}\left( \left\vert
Du\right\vert \right) \right) ^{\gamma }\,dx.
\end{align*}%
We sum over $k$ and, since the right hand side does not depend on $k$, we
get 
\begin{align*}
& \sum_{k=1}^{n}\left\vert \int_{\Omega }2\eta \Phi \left( \left\vert
Du\right\vert \right) \sum_{i=1}^{n}f_{\xi _{i}x_{k}}\left( x,Du\right) \eta
_{x_{i}}u_{x_{k}}\,dx\right\vert \\
& \quad \leq 2nM\int_{\Omega }\eta \left\vert D\eta \right\vert \Phi \left(
\left\vert Du\right\vert \right) \left\vert Du\right\vert \left(
1+\left\vert Du\right\vert ^{\gamma }\right) \left( g_{2}\left( \left\vert
Du\right\vert \right) \right) ^{\gamma }\,dx.
\end{align*}

\subsubsection{Estimate of the term in $\protect\fbox{\textbf{5}}$\label%
{Subsection: Estimate of the term in 5}}

We use the growth assumptions \eqref{growth conditions A},%
\eqref{Assumption
3} to get 
\begin{align*}
& \left\vert \int_{\Omega }\eta ^{2}\Phi \left( \left\vert Du\right\vert
\right) \sum_{i=1}^{n}f_{\xi _{i}x_{k}}\left( x,Du\right)
u_{x_{k}x_{i}}\,dx\right\vert \\
& \quad \leq \int_{\Omega }\eta ^{2}\Phi \left( \left\vert Du\right\vert
\right) g_{3}\left( \left\vert Du\right\vert \right) \left\vert
Du_{x_{k}}\right\vert \,dx \\
& \quad \leq M\int_{\Omega }\eta ^{2}\Phi \left( \left\vert Du\right\vert
\right) \left( g_{1}\left( \left\vert Du\right\vert \right) \right) ^{\frac{1%
}{2}}\left( 1+\left\vert Du\right\vert ^{\gamma }\right) \left( g_{1}\left(
\left\vert Du\right\vert \right) \right) ^{\frac{\gamma -1}{2}}\left(
g_{2}\left( \left\vert Du\right\vert \right) \right) ^{\frac{\gamma }{2}%
}\left\vert Du_{x_{k}}\right\vert \,dx \\
& \quad \leq M\int_{\Omega }\eta ^{2}\Phi \left( \left\vert Du\right\vert
\right) \left( g_{1}\left( \left\vert Du\right\vert \right) \right) ^{\frac{1%
}{2}}\left( 1+\left\vert Du\right\vert ^{\gamma }\right) \left( g_{2}\left(
\left\vert Du\right\vert \right) \right) ^{\gamma -\frac{1}{2}}\left\vert
Du_{x_{k}}\right\vert \,dx \\
& \quad \leq \varepsilon M\int_{\Omega }\eta ^{2}\Phi \left( \left\vert
Du\right\vert \right) g_{1}\left( \left\vert Du\right\vert \right)
\,\left\vert Du_{x_{k}}\right\vert ^{2}\,dx \\
& \quad \quad +\frac{M}{4\varepsilon }\int_{\Omega }\eta ^{2}\Phi \left(
\left\vert Du\right\vert \right) \left( 1+\left\vert Du\right\vert ^{\gamma
}\right) ^{2}\left( g_{2}\left( \left\vert Du\right\vert \right) \right)
^{2\gamma -1}\,dx\,.
\end{align*}%
We sum over $k$ and we observe that $\sum_{k=1}^{n}\left\vert
Du_{x_{k}}\right\vert ^{2}=\left\vert D^{2}u\right\vert ^{2}$ 
\begin{equation*}
\sum_{k=1}^{n}\left\vert \int_{\Omega }\eta ^{2}\Phi \left( \left\vert
Du\right\vert \right) \sum_{i=1}^{n}f_{\xi _{i}x_{k}}\left( x,Du\right)
u_{x_{k}x_{i}}\,dx\right\vert
\end{equation*}%
\begin{equation*}
\leq \varepsilon M\int_{\Omega }\eta ^{2}\Phi \left( \left\vert
Du\right\vert \right) g_{1}\left( \left\vert Du\right\vert \right)
\,\left\vert D^{2}u\right\vert ^{2}\,dx
\end{equation*}%
\begin{equation*}
+\frac{nM}{4\varepsilon }\int_{\Omega }\eta ^{2}\Phi \left( \left\vert
Du\right\vert \right) \left( 1+\left\vert Du\right\vert ^{\gamma }\right)
^{2}\left( g_{2}\left( \left\vert Du\right\vert \right) \right) ^{2\gamma
-1}\,dx\,.
\end{equation*}

\subsubsection{Estimate of the term in $\protect\fbox{\textbf{6}}$}

Similarly to the previous Sections, under the notation $D\left( \left\vert
Du\right\vert \right) =\left( \left\vert Du\right\vert _{x_{i}}\right)
_{i=1,2,\ldots ,n}$ and $\sum_{i=1}^{n}\left( \left\vert Du\right\vert
_{x_{i}}\right) ^{2}=\left\vert D\left( \left\vert Du\right\vert \right)
\right\vert ^{2}$, by the Cauchy-Schwarz inequality and 
\eqref{growth
conditions A}, \eqref{Assumption 3}, we obtain

\begin{equation}
\left\vert \int_{\Omega }\eta ^{2}\Phi ^{\prime }\left( \left\vert
Du\right\vert \right) \sum_{i=1}^{n}f_{\xi _{i}x_{k}}\left( x,Du\right)
u_{x_{k}}\left\vert Du\right\vert _{x_{i}}\,dx\right\vert
\label{Estimate of the term in 6 a}
\end{equation}%
\begin{equation*}
\int_{\Omega }\eta ^{2}\Phi ^{\prime }\left( \left\vert Du\right\vert
\right) \left\vert u_{x_{k}}\right\vert (\sum_{i=1}^{n}\left( f_{\xi
_{i}x_{k}}\left( x,Du\right) \right) ^{2})^{\frac{1}{2}}(\sum_{i=1}^{n}%
\left( \left\vert Du\right\vert _{x_{i}}\right) ^{2})^{\frac{1}{2}}\,dx
\end{equation*}%
\begin{equation*}
\leq \sqrt{n}\int_{\Omega }\eta ^{2}\Phi ^{\prime }\left( \left\vert
Du\right\vert \right) \left\vert Du\right\vert g_{3}\left( \left\vert
Du\right\vert \right) \,\left\vert D\left( \left\vert Du\right\vert \right)
\right\vert \,dx
\end{equation*}%
\begin{equation*}
\leq \sqrt{n}M\int_{\Omega }\eta ^{2}\Phi ^{\prime }\left( \left\vert
Du\right\vert \right) \left\vert Du\right\vert \left( g_{1}\left( \left\vert
Du\right\vert \right) \right) ^{\frac{1}{2}}\left( 1+\left\vert
Du\right\vert ^{\gamma }\right) \left( g_{1}\left( \left\vert Du\right\vert
\right) \right) ^{\frac{\gamma -1}{2}}\left( g_{2}\left( \left\vert
Du\right\vert \right) \right) ^{\frac{\gamma }{2}}\,\left\vert D\left(
\left\vert Du\right\vert \right) \right\vert \,dx
\end{equation*}%
\begin{equation*}
\leq \varepsilon \sqrt{n}M\int_{\Omega }\eta ^{2}\Phi ^{\prime }\left(
\left\vert Du\right\vert \right) \left\vert Du\right\vert g_{1}\left(
\left\vert Du\right\vert \right) \,\left\vert D\left( \left\vert
Du\right\vert \right) \right\vert ^{2}\,dx
\end{equation*}%
\begin{equation*}
+\frac{\sqrt{n}M}{4\varepsilon }\int_{\Omega }\eta ^{2}\Phi ^{\prime }\left(
\left\vert Du\right\vert \right) \left\vert Du\right\vert \left(
1+\left\vert Du\right\vert ^{\gamma }\right) ^{2}\left( g_{2}\left(
\left\vert Du\right\vert \right) \right) ^{2\gamma -1}\,dx\,.
\end{equation*}

\subsection{Step 3. Collecting together the previous estimates}

First we make the previous estimates uniform each other with a sum over $%
k\in \left\{ 1,2,\ldots ,n\right\} $ where this has been not jet done. Then
we collect them, starting from the estimate of the term in $\fbox{\textbf{1}}
$ with the inequality (\ref{Estimate of the term in 1 bis}), with the
nonnegative terms $\fbox{\textbf{2}}$ and $\fbox{\textbf{3}}$ in the left
hand side. By also taking the $\varepsilon -$addenda of $\fbox{\textbf{5}}$
and $\fbox{\textbf{6}}$ in the left hand side, we obtain 
\begin{align*}
& (\underset{\text{{\tiny \fbox{\textbf{1}}}}}{\tfrac{1}{2}}-\underset{%
{\tiny \fbox{\textbf{5}}}}{\varepsilon M})\int_{\Omega }\underset{\text{%
{\tiny \fbox{\textbf{2}}}}}{\eta ^{2}\Phi \left( \left\vert Du\right\vert
\right) g_{1}\left( \left\vert Du\right\vert \right) \,\left\vert
D^{2}u\right\vert ^{2}\,dx} \\
& +(1-\underset{{\tiny \fbox{\textbf{6}}}}{\varepsilon \sqrt{n}M}%
)\int_{\Omega }\underset{\text{{\tiny \fbox{\textbf{3}}}}}{\eta ^{2}\Phi
^{\prime }\left( \left\vert Du\right\vert \right) \left\vert Du\right\vert
g_{1}\left( \left\vert Du\right\vert \right) \,\left\vert D\left( \left\vert
Du\right\vert \right) \right\vert ^{2}\,dx} \\
& \quad \leq 2nM\int_{\Omega }\underset{{\tiny \fbox{\textbf{4}}}}{\eta
\left\vert D\eta \right\vert \Phi \left( \left\vert Du\right\vert \right)
\left\vert Du\right\vert \left( 1+\left\vert Du\right\vert ^{\gamma }\right)
\left( g_{2}\left( \left\vert Du\right\vert \right) \right) ^{\gamma }}\,dx
\\
& \quad \quad +\frac{nM}{4\varepsilon }\int_{\Omega }\underset{{\tiny \fbox{%
\textbf{5}}}}{\eta ^{2}\Phi \left( \left\vert Du\right\vert \right) \left(
1+\left\vert Du\right\vert ^{\gamma }\right) ^{2}\left( g_{2}\left(
\left\vert Du\right\vert \right) \right) ^{2\gamma -1}\,dx} \\
& \quad \quad +\frac{\sqrt{n}M}{4\varepsilon }\int_{\Omega }\underset{{\tiny 
\fbox{\textbf{6}}}}{\eta ^{2}\Phi ^{\prime }\left( \left\vert Du\right\vert
\right) \left\vert Du\right\vert \left( 1+\left\vert Du\right\vert ^{\gamma
}\right) ^{2}\left( g_{2}\left( \left\vert Du\right\vert \right) \right)
^{2\gamma -1}\,dx} \\
& \quad \quad +2\int_{\Omega }\underset{\text{{\tiny \fbox{\textbf{1}}}}}{%
\left\vert D\eta \right\vert ^{2}\Phi \left( \left\vert Du\right\vert
\right) \left\vert Du\right\vert ^{2}g_{2}\left( \left\vert Du\right\vert
\right) \,dx}\,.
\end{align*}%
We chose $\varepsilon =\frac{1}{4M}\min \left\{ 1;\frac{3}{\sqrt{n}}\right\} 
$, so that $\frac{1}{2}-\varepsilon M\geq \frac{1}{4}$ and $1-\varepsilon 
\sqrt{n}M\geq \frac{1}{4}$ too. Form the previous estimate we obtain the
existence of a positive constant $c_{0}=c_{0}\left( n,M\right) $ such that 
\begin{equation*}
\tfrac{1}{c_{0}}\int_{\Omega }\eta ^{2}g_{1}\left( \left\vert Du\right\vert
\right) \left( \Phi \left( \left\vert Du\right\vert \right) \,\left\vert
D^{2}u\right\vert ^{2}+\Phi ^{\prime }\left( \left\vert Du\right\vert
\right) \left\vert Du\right\vert \,\left\vert D\left( \left\vert
Du\right\vert \right) \right\vert ^{2}\right) \,dx
\end{equation*}%
\begin{equation*}
\leq \int_{\Omega }\eta \left\vert D\eta \right\vert \Phi \left( \left\vert
Du\right\vert \right) \left\vert Du\right\vert \left( 1+\left\vert
Du\right\vert ^{\gamma }\right) \left( g_{2}\left( \left\vert Du\right\vert
\right) \right) ^{\gamma }\,dx
\end{equation*}%
\begin{equation*}
+\int_{\Omega }\eta ^{2}\left( \Phi \left( \left\vert Du\right\vert \right)
+\Phi ^{\prime }\left( \left\vert Du\right\vert \right) \left\vert
Du\right\vert \right) \left( 1+\left\vert Du\right\vert ^{\gamma }\right)
^{2}\left( g_{2}\left( \left\vert Du\right\vert \right) \right) ^{2\gamma
-1}\,dx
\end{equation*}%
\begin{equation*}
+\int_{\Omega }\left\vert D\eta \right\vert ^{2}\Phi \left( \left\vert
Du\right\vert \right) \left\vert Du\right\vert ^{2}g_{2}\left( \left\vert
Du\right\vert \right) \,dx\,.
\end{equation*}%
We make use of the inequality (\ref{gradient of modulus of Du 4}); i.e., $%
\left\vert D\left( \left\vert Du\right\vert \right) \right\vert \leq
\left\vert D^{2}u\right\vert $. Therefore finally we get 
\begin{equation}
\tfrac{1}{c_{0}}\int_{\Omega }\eta ^{2}g_{1}\left( \left\vert Du\right\vert
\right) \left( \Phi \left( \left\vert Du\right\vert \right) +\Phi ^{\prime
}\left( \left\vert Du\right\vert \right) \left\vert Du\right\vert \right)
\left\vert D\left( \left\vert Du\right\vert \right) \right\vert ^{2}\,dx
\label{25M}
\end{equation}%
\begin{equation*}
\leq \int_{\Omega }\eta \left\vert D\eta \right\vert \Phi \left( \left\vert
Du\right\vert \right) \left\vert Du\right\vert \left( 1+\left\vert
Du\right\vert ^{\gamma }\right) \left( g_{2}\left( \left\vert Du\right\vert
\right) \right) ^{\gamma }\,dx
\end{equation*}%
\begin{equation*}
+\int_{\Omega }\eta ^{2}\left( \Phi \left( \left\vert Du\right\vert \right)
+\Phi ^{\prime }\left( \left\vert Du\right\vert \right) \left\vert
Du\right\vert \right) \left( 1+\left\vert Du\right\vert ^{\gamma }\right)
^{2}\left( g_{2}\left( \left\vert Du\right\vert \right) \right) ^{2\gamma
-1}\,dx
\end{equation*}%
\begin{equation*}
+\int_{\Omega }\left\vert D\eta \right\vert ^{2}\Phi \left( \left\vert
Du\right\vert \right) \left\vert Du\right\vert ^{2}g_{2}\left( \left\vert
Du\right\vert \right) \,dx\,.
\end{equation*}

\subsection{Step 4. Use of the Sobolev inequality}

We define the integral function $G:\left[ 0,+\infty \right) \rightarrow %
\left[ 1,+\infty \right) $ by 
\begin{equation}
G(t)=1+\int_{0}^{t}\sqrt{\Phi (s)g_{1}(s)}\,ds.  \label{G}
\end{equation}%
Since $g_{1}$ and $\Phi $ are increasing, then $\int_{0}^{t}\sqrt{\Phi
(s)g_{1}(s)}\leq \sqrt{\Phi (t)g_{1}(t)}\cdot t$ and 
\begin{equation*}
\lbrack G(t)]^{2}\leq 2\left( 1+\left( \int_{0}^{t}\sqrt{\Phi (s)g_{1}(s)}%
\right) ^{2}\right) \leq 2\left( 1+\Phi (t)g_{1}(t)t^{2}\right) \,.
\end{equation*}%
We deduce 
\begin{equation*}
\left\vert D(\eta G(|Du|)\right\vert ^{2}\leq 2|D\eta
|^{2}[G(|Du|)]^{2}+2\eta ^{2}[G^{\prime }(|Du|)]^{2}|D(|Du|)|^{2}
\end{equation*}%
\begin{equation*}
\leq 4|D\eta |^{2}(1+\Phi (|Du|)g_{1}(|Du|)|Du|^{2})+2\eta ^{2}\Phi
(|Du|)g_{1}(|Du|)\,|D(|Du|)|^{2}.
\end{equation*}%
By integrating we get 
\begin{equation}
\int_{\Omega }|D(\eta G(|Du|)|^{2}\,dx\leq 4\int_{\Omega }|D\eta
|^{2}(1+\Phi (|Du|)g_{1}(|Du|)|Du|^{2})\,dx  \label{26M}
\end{equation}%
\begin{equation*}
+2\int_{\Omega }\eta ^{2}\Phi (|Du|)g_{1}(|Du|)\,|D(|Du|)|^{2}\,dx\,.
\end{equation*}%
When compared with the estimate (\ref{25M}), we note that the last addendum
in the right hand side of (\ref{26M}), is less that or equal to the left
hand side of (\ref{25M}). By combining (\ref{25M}), (\ref{26M}) we obtain 
\begin{equation}
\int_{\Omega }|D(\eta G(|Du|)|^{2}\,dx\leq \int_{\Omega }|D\eta
|^{2}(4+2c_{0}\Phi (|Du|)g_{1}(|Du|)|Du|^{2})\,dx  \label{27M}
\end{equation}%
\begin{equation*}
+2c_{0}\int_{\Omega }\eta \left\vert D\eta \right\vert \Phi \left(
\left\vert Du\right\vert \right) \left\vert Du\right\vert \left(
1+\left\vert Du\right\vert ^{\gamma }\right) \left( g_{2}\left( \left\vert
Du\right\vert \right) \right) ^{\gamma }\,dx
\end{equation*}%
\begin{equation*}
+2c_{0}\int_{\Omega }\eta ^{2}\left( \Phi \left( \left\vert Du\right\vert
\right) +\Phi ^{\prime }\left( \left\vert Du\right\vert \right) \left\vert
Du\right\vert \right) \left( 1+\left\vert Du\right\vert ^{\gamma }\right)
^{2}\left( g_{2}\left( \left\vert Du\right\vert \right) \right) ^{2\gamma
-1}\,dx\,.
\end{equation*}%
Recall that $2^{\ast }=\frac{2n}{n-2}$ if $n\geq 3$, while $2^{\ast }$ is
any fixed real number greater than $\alpha $ if $n=2$. By the Sobolev
inequality, there exists a constant $c_{1}$ such that 
\begin{equation}
\left( \int_{\Omega }[\eta G(|Du|)]^{2^{\ast }}\,dx\right) ^{2/2^{\ast
}}\leq c_{1}\int_{\Omega }|D(\eta G(|Du|)|^{2}\,dx\,,  \label{28M}
\end{equation}%
which naturally can be combined with (\ref{27M}), as done in the next
section.

\subsection{Step 5. Choice of the test function $\Phi $\label{Section: Step
5}}

Let us define $\Phi (t)=t^{2\lambda }$ for every $t\geq 0$, with $\lambda
\geq 1$. Making use of \cite[Lemma 3.4 (v)]{Marcellini 1993} we can estimate
the integral function $G(t)$ defined in (\ref{G}) by 
\begin{equation}
G(t)=1+\int_{0}^{t}\sqrt{\Phi (s)g_{1}(s)}\,ds\geq 1+\frac{t^{\lambda }}{%
\lambda +1}\int_{0}^{t}\sqrt{g_{1}(s)}\,ds.  \label{29M}
\end{equation}%
Recall that the function $g_{1}:[0,+\infty )\rightarrow \lbrack 0,+\infty )$
is increasing and not identically zero. Then, there exists $t_{0}>0$ such
that $g_{1}(t)>0$, for every $t\geq t_{0}$. Up to a rescaling, we can assume
that $t_{0}\leq 1$, so that 
\begin{equation}
g_{1}(1)=c_{2}>0,\quad \int_{0}^{1}\sqrt{g_{1}(s)}\,ds=c_{3}>0.  \label{30M}
\end{equation}%
Then, for every $t\geq 1$, we have 
\begin{align}
2\int_{0}^{t}\sqrt{g_{1}(s)}\,ds& \geq c_{3}+\int_{0}^{t}\sqrt{g_{1}(s)}\,ds
\notag  \label{31M} \\
& \geq \min \{c_{3},1\}\left( 1+\int_{0}^{t}\sqrt{g_{1}(s)}\,ds\right) .
\end{align}%
Now, we use assumptions \eqref{12M} and \eqref{condition on beta}. By %
\eqref{12M} and by \eqref{31M}, there exists a constant $c_{4}$ such that 
\begin{equation}
\left( \int_{0}^{t}\sqrt{g_{1}(s)}\,ds\right) ^{\alpha }\geq
c_{4}(g_{2}(t))^{2\gamma -1}t^{2},\quad \forall t\geq 1.  \label{32M}
\end{equation}%
Recall that by \eqref{condition on alpha}, $2\leq \alpha <2^{\ast }$. Then,
by \eqref{30M}, for every $t\geq 1$ we obtain 
\begin{align}
\left( \int_{0}^{t}\sqrt{g_{1}(s)}\,ds\right) ^{2^{\ast }}& \geq \left(
\int_{0}^{t}\sqrt{g_{1}(s)}\,ds\right) ^{\alpha }(c_{2}(t-1)+c_{3})^{2^{\ast
}-\alpha }  \notag  \label{33M} \\
& \geq (\min \{c_{2},c_{3}\})^{2^{\ast }-\alpha }c_{4}(g_{2}(t))^{2\gamma
-1}t^{2^{\ast }-\alpha +2},\quad \forall t\geq 1.
\end{align}%
We note that for the previous step, it's fundamental assuming $\alpha $
strictly less then $2^{\ast }$. By \eqref{33M} and \eqref{29M}, we deduce
that for every $t\geq 1$ 
\begin{align*}
\lbrack G(t)]^{2^{\ast }}& \geq \frac{1}{2}\left( 1+\left( \frac{t^{\lambda }%
}{(\lambda +1)}\int_{0}^{t}\sqrt{g_{1}(s)}\,ds\right) ^{2^{\ast }}\right)  \\
& \geq \frac{1}{2}+\frac{(\min \{c_{2},c_{3}\})^{2^{\ast }-\alpha }c_{4}}{%
2(\lambda +1)^{2}}(g_{2}(t))^{2\gamma -1}t^{2^{\ast }(\lambda +1)-\alpha +2}.
\end{align*}%
That is, there exists a constant $c_{5}$ such that 
\begin{equation*}
c_{5}(\lambda +1)^{2^{\ast }}[G(t)]^{2^{\ast }}\geq 1+(g_{2}(t))^{2\gamma
-1}t^{2^{\ast }(\lambda +1)-\alpha +2},
\end{equation*}%
for every $t\geq 1$, and also for $t\in \lbrack 0,1)$, since the left-hand
side is bounded for $t\in \left( 0,1\right) $ by the constant (independent
of $\gamma $) $1+g_{2}(1)$, while the right-hand side is bounded from below
away from zero. So, for every $t\geq 0$, 
\begin{equation}
c_{5}(\lambda +1)^{2^{\ast }}[G(t)]^{2^{\ast }}\geq 1+(g_{2}(t))^{2\gamma
-1}t^{2^{\ast }(\lambda +1)-\alpha +2}\,.  \label{34M}
\end{equation}%
Since $\Phi (t)=t^{2\lambda }$ and $g_{1}(t)\leq g_{2}(t)$, from \eqref{27M}%
, \eqref{28M} and \eqref{34M}, with positive constants $c_{6}$, $c_{7}$, $%
\ldots $, we obtain 
\begin{equation*}
\left( \int_{\Omega }\eta ^{2^{\ast }}\left\{ 1+\left\vert Du\right\vert
^{2^{\ast }(\lambda +1)-\alpha +2}(g_{2}(\left\vert Du\right\vert
))^{2\gamma -1}\right\} \,dx\right) ^{2/2^{\ast }}
\end{equation*}%
\begin{equation*}
\leq c_{6}\left( \lambda +1\right) ^{2}\left( \int_{\Omega }\eta ^{2^{\ast
}}[G(\left\vert Du\right\vert )]^{2^{\ast }}\,dx\right) ^{2/2^{\ast }}
\end{equation*}%
\begin{equation*}
\leq c_{7}\left( \lambda +1\right) ^{2}\left\{ \int_{\Omega }|D\eta
|^{2}(1+\left\vert Du\right\vert ^{2\lambda }g_{2}(\left\vert Du\right\vert
)\,\left\vert Du\right\vert ^{2})\,dx\right. 
\end{equation*}%
\begin{equation*}
+\int_{\Omega }\eta \left\vert D\eta \right\vert \left\vert Du\right\vert
^{2\lambda +1}\left( 1+\left\vert Du\right\vert ^{\gamma }\right) \left(
g_{2}\left( \left\vert Du\right\vert \right) \right) ^{\gamma }\,dx
\end{equation*}%
\begin{equation*}
+\left. \int_{\Omega }\eta ^{2}\left( |Du|^{2\lambda }+2\lambda
|Du|^{2\lambda }\right) \left( 1+\left\vert Du\right\vert ^{\gamma }\right)
^{2}\left( g_{2}\left( \left\vert Du\right\vert \right) \right) ^{2\gamma
-1}\,dx\right\} \,.
\end{equation*}%
Then we uniform as much as possible the right had side 
\begin{equation}
\left( \int_{\Omega }\eta ^{2^{\ast }}\left\{ 1+\left\vert Du\right\vert
^{2^{\ast }(\lambda +1)-\alpha +2}(g_{2}(\left\vert Du\right\vert
))^{2\gamma -1}\right\} \,dx\right) ^{2/2^{\ast }}  \label{updated 1}
\end{equation}%
\begin{equation*}
\leq c_{8}\left( \lambda +1\right) ^{3}\left\{ \int_{\Omega }|D\eta
|^{2}(1+\left\vert Du\right\vert ^{2\lambda }g_{2}(\left\vert Du\right\vert
)\,\left\vert Du\right\vert ^{2})\,dx\right. 
\end{equation*}%
\begin{equation*}
+\int_{\Omega }\eta \left\vert D\eta \right\vert \left\vert Du\right\vert
^{2\lambda +1}\left( 1+\left\vert Du\right\vert ^{\gamma }\right) \left(
g_{2}\left( \left\vert Du\right\vert \right) \right) ^{\gamma }\,dx
\end{equation*}%
\begin{equation*}
+\left. \int_{\Omega }\eta ^{2}\left( |Du|^{2\lambda }+|Du|^{2\lambda
}\right) \left( 1+\left\vert Du\right\vert ^{\gamma }\right) ^{2}\left(
g_{2}\left( \left\vert Du\right\vert \right) \right) ^{2\gamma
-1}\,dx\right\} \,.
\end{equation*}%
With the aim to reduce the previous estimate to a unique addendum in the
right hand side we recall that $\gamma \geq 1$ and $\lambda \geq 1$.
Moreover $g_{2}:\left( 0,+\infty \right) \rightarrow \left( 0,+\infty
\right) $ is an increasing function such that $g_{2}\left( 1\right) \geq 1$,
and thus also $g_{2}\left( t\right) \geq 1$ for all $t\geq 1$. Thus for
every $t\geq 0$ we have 
\begin{align}
t^{2\lambda }\left( 1+t^{\gamma }\right) ^{2}\left( g_{2}\left( t\right)
\right) ^{2\gamma -1}& \leq \left\{ 
\begin{array}{l}
4\,t^{2\left( \lambda +\gamma \right) }\left( g_{2}\left( t\right) \right)
^{2\gamma -1}\hspace{0.88cm}\text{if}\;\;t\geq 1 \\ 
4\left( g_{2}\left( 1\right) \right) ^{2\gamma -1}\hspace{2cm}\text{if}%
\;\;0\leq t\leq 1%
\end{array}%
\right\}   \notag \\
& \leq \max \left\{ 4\,\left( g_{2}\left( 1\right) \right) ^{2\gamma
-1};4\,t^{2\left( \lambda +\gamma \right) }\left( g_{2}\left( t\right)
\right) ^{2\gamma -1}\right\}   \notag \\
& \leq 4\,\left( g_{2}\left( 1\right) \right) ^{2\gamma -1}+4\,t^{2\left(
\lambda +\gamma \right) }\left( g_{2}\left( t\right) \right) ^{2\gamma -1} 
\notag \\
& \leq 4\,\left( g_{2}\left( 1\right) \right) ^{2\gamma -1}+4\,\left(
g_{2}\left( 1\right) \right) ^{2\gamma -1}\,t^{2\left( \lambda +\gamma
\right) }\left( g_{2}\left( t\right) \right) ^{2\gamma -1}  \notag \\
& =4\,\left( g_{2}\left( 1\right) \right) ^{2\gamma -1}\left( 1+\,t^{2\left(
\lambda +\gamma \right) }\left( g_{2}\left( t\right) \right) ^{2\gamma
-1}\right) \,.  \label{updated 0}
\end{align}%
By proceeding in a similar way we get$_{{}}$ 
\begin{align}
t^{2\lambda +1}(1+t)^{\gamma }\left( g_{2}(t)\right) ^{\gamma }& \leq
\left\{ 
\begin{array}{l}
2\,t^{2\lambda +\gamma +1}\left( g_{2}(1)\right) ^{\gamma }\hspace{0.63cm}%
\text{if}\;\;t\geq 1 \\ 
2\left( g_{2}(1)\right) ^{\gamma }\hspace{1.85cm}\text{if}\;\;0\leq t\leq 1%
\end{array}%
\right\}   \notag \\
& \leq \left\{ 
\begin{array}{l}
2\,t^{2\left( \lambda +\gamma \right) }\left( g_{2}(1)\right) ^{\gamma }%
\hspace{0.7cm}\text{if}\;\;t\geq 1 \\ 
2\left( g_{2}(1)\right) ^{\gamma }\hspace{1.85cm}\text{if}\;\;0\leq t\leq 1%
\end{array}%
\right\}   \notag \\
& \leq 2\left( g_{2}\left( 1\right) \right) ^{2\gamma -1}\left( 1+t^{2\left(
\lambda +\gamma \right) }\left( g_{2}(t)\right) ^{2\gamma -1}\right)
\,,\;\;\;\forall \;t\geq 0\,,  \label{updated 5}
\end{align}%
and 
\begin{align}
t^{2(\lambda +1)}g_{2}(t)& \leq g_{2}(1)\left( 1+t^{2\left( \lambda +\gamma
\right) }\left( g_{2}(t)\right) ^{2\gamma -1}\right)   \notag \\
& \leq \left( g_{2}\left( 1\right) \right) ^{2\gamma -1}\left( 1+t^{2\left(
\lambda +\gamma \right) }\left( g_{2}(t)\right) ^{2\gamma -1}\right)
\,,\;\;\;\forall \;t\geq 0.  \label{updated 4}
\end{align}%
By using \eqref{updated 0}, \eqref{updated 5}, \eqref{updated 4}, from %
\eqref{updated 1} we get the final estimate with only one addendum in the
right hand side 
\begin{align}
& \left( \int_{\Omega }\eta ^{2^{\ast }}[1+|Du|^{2^{\ast }(\lambda
+1)-\alpha +2}(g_{2}(\left\vert Du\right\vert ))^{2\gamma -1}]\,dx\right)
^{2/2^{\ast }}  \notag  \label{updated 6} \\
& \quad \leq c(2\lambda +1)^{3}\int_{\Omega }\left( \eta +\left\vert D\eta
\right\vert \right) ^{2}\left( 1+\left\vert Du\right\vert ^{2\left( \lambda
+\gamma \right) }\left( g_{2}(\left\vert Du\right\vert )\right) ^{2\gamma
-1}\right) \,dx\,.
\end{align}%
Note that the constant $c$ in (\ref{updated 6}) depends on the dimension $n,$
on the other constants $m,M,\gamma $, on the values $g_{1}(1)$, $g_{2}(1)$,
but it is independent of $\lambda $.

Let us denote by $B_{R}$ and $B_{\rho }$ balls compactly contained in $%
\Omega $, of radii respectively $R$, $\rho $, with the same center. Let $%
\eta $ be a test function equal to $1$ in $B_{\rho }$, whose support is
contained in $B_{R}$, such that $\left\vert D\eta \right\vert \leq \frac{2}{%
R-\rho }$. Thus, we obtain 
\begin{align}
& \left( \int_{B_{\rho }}[1+|Du|^{2^{\ast }(\lambda +1)-\alpha
+2}(g_{2}(|Du|))^{2\gamma -1}]\,dx\right) ^{2/2^{\ast }}  \notag  \label{35M}
\\
& \quad \leq c(2\lambda +1)^{3}\int_{B_{R}}\left( \eta ^{2}+|D\eta
|^{2}\right) \left( 1+\left\vert Du\right\vert ^{2\left( \lambda +\gamma
\right) }\left( g_{2}(\left\vert Du\right\vert )\right) ^{2\gamma -1}\right)
\,dx  \notag \\
& \quad \leq c(2\lambda +1)^{3}\left( 1+\frac{4}{(R-\rho )^{2}}\right)
\int_{B_{R}}\left( 1+\left\vert Du\right\vert ^{2\left( \lambda +\gamma
\right) }\left( g_{2}(\left\vert Du\right\vert )\right) ^{2\gamma -1}\right)
\,dx\,.  \notag \\
&
\end{align}

\subsection{Step 6. Iteration}

We define by induction a sequence $\lambda _{k}$ in the following way: 
\begin{equation}
\lambda _{1}=0,\quad \lambda _{k+1}=\frac{2^{\ast }}{2}\lambda _{k}+\frac{%
2^{\ast }-\alpha +2}{2}-\gamma ,\quad \forall k\in \mathbb{N};  \label{36M}
\end{equation}%
we note in particular that $\lambda _{k}$ satisfies the property 
\begin{equation}
2^{\ast }(\lambda _{k}+1)-\alpha +2=2(\lambda _{k+1}+\gamma ),\quad \quad
\forall k\in \mathbb{N}.  \label{37M}
\end{equation}%
It is easy to prove by induction the following representation formula for $%
\lambda _{k}$: 
\begin{equation}
\lambda _{k}=\frac{2^{\ast }-\alpha -2(\gamma -1)}{2^{\ast }-2}[(\frac{%
2^{\ast }}{2})^{k-1}-1],\quad \quad \forall k\in \mathbb{N}.  \label{38M}
\end{equation}%
Since $\alpha \geq 2$, from \eqref{37M}, \eqref{38M}, we deduce the
inequality 
\begin{equation}
(\lambda _{k}+1)2^{\ast }-\alpha +2\geq 2\,\frac{2^{\ast }-a-2(\gamma -1)}{%
2^{\ast }-2}\,\left( \frac{2^{\ast }}{2}\right) ^{k}.  \label{39M}
\end{equation}%
For fixed $R_{0}$ and $\rho _{0}$, for all $k\in \mathbb{N}$, we rewrite %
\eqref{35M} with $R=\rho _{k-1}$ and $\rho =\rho _{k}$, where $\rho
_{k}=\rho _{0}+\tfrac{R_{0}-\rho _{0}}{2^{k}}$; moreover, for $k=1,2,3,...,i$%
, with $i$ fixed in $\mathbb{N}$, we put $\lambda $ equal to $\lambda _{k}$.
By iterating \eqref{35M}, by \eqref{37M} we obtain 
\begin{align}
& \left( \int_{B_{\rho _{i}}}\left( 1+|Du|^{(\lambda _{i+1})2^{\ast
}-a+2}\right) \,\left( g_{2}(\left\vert Du\right\vert )\right) ^{2\gamma
-1}\,dx\right) ^{\frac{2}{2^{\ast }}}  \notag  \label{40M} \\
& \leq c_{7}\int_{B_{R_{0}}}\left( 1+|Du|^{2\gamma }\,\left(
g_{2}(\left\vert Du\right\vert )\right) ^{2\gamma -1}\right) \,dx.
\end{align}%
Since $R-\rho =\rho _{k-1}-\rho _{k}=\tfrac{R_{0}-\rho _{0}}{2^{k}}$ for all 
$k\in \mathbb{N}$, if $n\geq 3$; otherwise, if $n=2$, then for every $%
\epsilon >0$ we can choose $2^{\ast }$ so that 
\begin{equation}
c_{7}=\frac{c_{8}}{(R_{0}-\rho _{0})^{-2-\epsilon }},\quad 
\mbox{for some constant
$c_8$.}  \label{41M}
\end{equation}%
Since $g_{2}$ is increasing, then $g_{2}(t)\geq g_{1}(t)\geq g_{1}(1)>0$,
for all $t\geq 1$. Therefore for $r\geq s\geq 0$, we have 
\begin{equation*}
g_{2}(t)t^{r}+1\geq g_{1}(1)t^{s},\quad \mbox{if $t\geq 1,$}
\end{equation*}%
\begin{equation*}
g_{2}(t)t^{r}+1\geq 1\geq t^{s},\quad \mbox{if $0\leq t\leq 1.$}
\end{equation*}%
Thus, by posing $c_{9}=\min \{g_{1}(1),1\}$, we obtain 
\begin{equation}
g_{2}(t)t^{r}+1\geq c_{9}t^{s},\quad 
\mbox{$\forall t\geq 0$, $\forall r\geq s
\geq 0.$}  \label{42M}
\end{equation}%
Now, we go to the limit in \eqref{40M} as $i\rightarrow +\infty $. We use
the inequalities \eqref{39M}, \eqref{40M}, \eqref{42M} and we obtain 
\begin{align*}
& \sup \{|Du(x)|^{2\left( \frac{2^{\ast }-a-2(\gamma -1)}{2^{\ast }-2}%
\right) }:x\in B_{\rho _{0}}\} \\
& \quad =\lim_{i\rightarrow +\infty }\left( \int_{B_{\rho
_{0}}}|Du|^{2\left( \frac{2^{\ast }-a-2(\gamma -1)}{2^{\ast }-2}\right)
\left( \frac{2^{\ast }}{2}\right) ^{i}}\,dx\right) ^{\left( \frac{2}{2^{\ast
}}\right) ^{i}} \\
& \quad \leq \limsup_{i\rightarrow +\infty }\left( \frac{1}{c_{9}}%
\int_{B_{\rho _{i}}}\left( 1+|Du|^{2^{\ast }(\lambda _{i}+1)-a+2}\left(
g_{2}(\left\vert Du\right\vert )\right) ^{2\gamma -1}\right) \,dx\right)
^{\left( \frac{2}{2^{\ast }}\right) ^{i}} \\
& \quad \leq \limsup_{i\rightarrow +\infty }\frac{c_{7}}{c_{9}}%
\int_{B_{R_{0}}}\left( 1+|Du|^{2\gamma }\left( g_{2}(\left\vert
Du\right\vert )\right) ^{2\gamma -1}\right) \,dx
\end{align*}%
and by the representation of $c_{7}$ in \eqref{41M}, we finally obtain 
\begin{equation}
\Vert Du\Vert _{L^{\infty }(B_{\rho _{0}},\mathbb{R}^{n})}^{2\left( \frac{%
2^{\ast }-a-2(\gamma -1)}{2^{\ast }-2}\right) }\leq \frac{c_{10}}{%
(R_{0}-\rho _{0})^{n}}\int_{B_{R_{0}}}\left( 1+|Du|^{2\gamma }\left(
g_{2}(\left\vert Du\right\vert )\right) ^{2\gamma -1}\right) \,dx.
\label{43M}
\end{equation}

\subsection{Step 7. A priori gradient estimate}

We start putting together the inequalities \eqref{updated 6}, \eqref{27M}
and \eqref{28M}, with $\lambda =0$, that is $\Phi=1$, we get 
\begin{align}  \label{ineq1Step7}
&\left( \int_{\Omega }\left(\eta G(|Du|)\right)^{2^{\ast}}\,dx\right)
^{2/2^{\ast }}  \notag \\
&\quad \leq c\int_{\Omega }\left( \eta^2 +\left| D\eta \right|^2 \right)
\left( 1+\left| Du\right| ^{2\gamma }\left( g_{2}(\left| Du\right| )\right)
^{2\gamma -1}\right) \,dx\,.
\end{align}
We note that, by assumption \eqref{condition on alpha}, $\frac{2^*}{%
\alpha-2+2 \gamma}>1$. Then we can define $\nu$ such that 
\begin{equation}  \label{nu}
1\leq \nu < \frac{2^*}{\alpha-2+2 \gamma}.
\end{equation}

\noindent Following the lines of Step 5, we have, for every $t\geq 1$ 
\begin{align}
\left( \int_{0}^{t}\sqrt{g_{1}(s)}\,ds\right) ^{\frac{2^{\ast}}{\nu}}& \geq
\left( \int_{0}^{t}\sqrt{g_{1}(s)}\,ds\right) ^{\alpha }(c_{2}(t-1)+c_{3})^{ 
\frac{2^{\ast }}{\nu}-\alpha}  \notag  \label{33Mbis} \\
& \geq (\min \{c_{2},c_{3}\})^{\frac{2^{\ast }}{\nu}-\alpha
}c_{4}(g_{2}(t))^{2\gamma -1}t^{\frac{2^{\ast }}{\nu}-\alpha +2}.
\end{align}%
By \eqref{33Mbis} and \eqref{29M}, we deduce that for every $t\geq 1$ 
\begin{align*}
[G(t)]^{\frac{2^{\ast }}{\nu}}& \geq \frac{1}{2} \left(1+\left(\int_{0}^{t}%
\sqrt{g_{1(s)}}\,ds\right)^{\frac{2^{\ast }}{\nu}}\right) \\
& \geq \frac{1}{2}+\frac{(\min \{c_{2},c_{3}\})^{\frac{2^{\ast }}{\nu}%
-\alpha }c_{4}}{2}(g_{2}(t))^{2\gamma -1}t^{\frac{2^{\ast }}{\nu}-\alpha +2}.
\end{align*}%
That is, there exists a constant $c$ such that 
\begin{equation*}
[G(t)]^{\frac{2^{\ast }}{\nu}}\geq c \left(1+(g_{2}(t))^{2\gamma -1}t^{\frac{%
2^{\ast }}{\nu}-\alpha +2}\right),
\end{equation*}%
for every $t\geq 1$, and also for $t\in \lbrack 0,1)$, since the left-hand
side is bounded for $t\in \left( 0,1\right) $ by the constant (independent
of $\gamma $) $1+g_{2}(1)$, while the right-hand side is bounded from below
away from zero. So, for every $t\geq 0$, 
\begin{equation}
[G(t)]^{\frac{2^{\ast }}{\nu}}\geq c\left( 1+(g_{2}(t))^{2\gamma -1}t^{\frac{%
2^{\ast }}{\nu}-\alpha +2}\right)\,.  \label{34Mbis}
\end{equation}

Since $\nu < \frac{2^{\ast }}{\alpha -2+2\gamma }$, we have $\frac{2^{\ast }%
}{\nu }-\alpha +2> 2\gamma $; by \eqref{ineq1Step7}, \eqref{34Mbis}, we get 
\begin{align}
& \left( \int_{\Omega }\eta ^{2^{\ast }}[1+|Du|^{2\gamma}(g_{2}(\left|
Du\right| ))^{2\gamma -1}]^\nu\,dx\right) ^{2/2^{\ast }}  \notag \\
&\quad\leq \left( \int_{\Omega }\eta ^{2^{\ast }}[1+|Du|^{\frac{2^{\ast }}{%
\nu }-\alpha +2}(g_{2}(\left| Du\right| ))^{2\gamma -1}]^\nu\,dx\right)
^{2/2^{\ast }}  \notag  \label{4.43DMM} \\
& \quad \leq c\int_{\Omega }\left( \eta ^{2}+\left| D\eta \right|
^{2}\right) \left( 1+\left| Du\right| ^{2\gamma }\left( g_{2}(\left|
Du\right| )\right) ^{2\gamma -1}\right) \,dx\,.
\end{align}
Then, under the notation 
\begin{equation*}
V=V(x)=1+\left| Du\right| ^{2\gamma }\left( g_{2}(\left| Du\right| )\right)
^{2\gamma -1},
\end{equation*}
inequality \eqref{4.43DMM} becomes 
\begin{equation}
\left( \int_{\Omega }\eta ^{2^{\ast }}V^{\nu }\,dx\right) ^{2/2^{\ast }}\leq
c\int_{\Omega }\left( \eta ^{2}+\left| D\eta \right| ^{2}\right) V\,dx\,.
\label{4.44DMM}
\end{equation}

We consider a test function $\eta =1$ in $B_{\rho }$ with $|D\eta |\leq 
\frac{2}{R-\rho }$ as in the previous step, thus we get 
\begin{equation}
\left( \int_{B_{\rho }}V^{\nu }\,dx\right) ^{2/2^{\ast }}\leq \frac{c}{%
(R-\rho )^{2}}\int_{B_{R}}V\,dx\,.  \label{4.45DMM}
\end{equation}%
Let $\mu >\frac{2^{\ast }}{2}$ to be chosen later. By H\"{o}lder inequality
we have 
\begin{align}
\left( \int_{B_{\rho }}V^{\nu }\,dx\right) ^{2/2^{\ast }}& \leq \frac{c}{%
(R-\rho )^{2}}\int_{B_{R}}V^{\frac{\nu }{\mu }}V^{1-\frac{\nu }{\mu }}\,dx\,
\notag  \label{4.46DMM} \\
& \leq \left( \int_{B_{R}}V^{\nu }\,dx\right) ^{\frac{1}{\mu }}\left(
\int_{B_{R}}V^{\frac{\mu -\nu }{\mu -1}}\,dx\right) ^{\frac{\mu -1}{\mu }}.
\end{align}%
Let $R_{0}$ and $\rho _{0}$ be fixed. For any $i\in \mathbb{N}$ we consider %
\eqref{4.46DMM} with $R=\rho _{i}$ and $\rho =\rho _{i-1}$, where $\rho
_{i}=R_{0}-\frac{R_{0}-\rho _{0}}{2^{i}}$. By iterating \eqref{4.46DMM}
since $R-\rho =\frac{R_{0}-\rho _{0}}{2^{i}}$, we obtain 
\begin{align}
\int_{B_{\rho _{0}}}V^{\nu }\,dx& \leq \left( \int_{B_{\rho _{i}}}V^{\nu
}\,dx\right) ^{\left( \frac{2^{\ast }}{2\mu }\right)
^{i}}\prod_{i=1}^{\infty }\left( \frac{c4^{i+1}}{(R_{0}-\rho _{0})^{2}}%
\right) ^{\mu \left( \frac{2^{\ast }}{2\mu }\right) ^{i}}  \notag
\label{47M} \\
& \hspace{4cm}\cdot \left( \int_{B_{\rho _{0}}}V^{\frac{\mu -\nu }{\mu -1}%
}\,dx\right) ^{(\mu -1)\left( \frac{2^{\ast }}{2\mu }\right) ^{i}}  \notag \\
& \leq \left( \int_{B_{\rho _{i}}}V^{\nu }\,dx\right) ^{\left( \frac{2^{\ast
}}{2\mu }\right) ^{i}}\,c\,\left( \frac{1}{(R_{0}-\rho _{0})^{2}}\right) ^{%
\frac{2^{\ast }\mu }{2\mu -2^{\ast }}}  \notag \\
& \hspace{4cm}\cdot \left( \int_{B_{\rho _{0}}}V^{\frac{\mu -\nu }{\mu -1}%
}\,dx\right) ^{2^{\ast }\frac{\mu -1}{2\mu -2^{\ast }}}.
\end{align}%
Now we use assumption \eqref{12M}: 
\begin{equation}
V=1+|Du|^{2\gamma }\left( g_{2}(|Du|)\right) ^{2\gamma -1}\leq 2\max
\{1,c\}(1+f(Du))^{\beta },  \label{48M}
\end{equation}%
where, since $\nu \geq 1$, 
\begin{equation}
\beta =\frac{\mu -1}{\mu -\nu }\geq 1.  \label{49M}
\end{equation}%
We recall that \eqref{nu} holds, so we have 
\begin{equation*}
1\leq \beta <\frac{\mu -1}{\mu -\frac{2^{\ast }}{\alpha -2+2\gamma }}.
\end{equation*}%
We compute the limit as $\mu \rightarrow \frac{2^{\ast }}{2}$ of the right
hand side, which, by computations, is equal to $\frac{2(\alpha -2+2\gamma )}{%
n(\alpha -4+2\gamma )}$. While the limit as $\mu \rightarrow +\infty $ is
equal to $1$. Then it is possible to choose $\mu \in \left( \frac{2^{\ast }}{%
2},+\infty \right) $ so that the definition of $\beta $ in \eqref{49M} is
compatible with assumption \eqref{condition on beta}.

\noindent We go to the limit in \eqref{47M} as $i\rightarrow +\infty $ and
we use \eqref{48M} and \eqref{49M} to obtain 
\begin{equation*}
\int_{B_{\rho _{0}}}V^{\nu }\,dx\leq c\left( \frac{1}{(R_{0}-\rho _{0})^{2}}%
\right) ^{\frac{2^{\ast }\mu }{2\mu -2^{\ast }}}\left(
\int_{B_{R_{0}}}(1+f(Du))\,dx\right) ^{\frac{2^{\ast }(\mu -1)}{2\mu
-2^{\ast }}}.
\end{equation*}%
Then we get 
\begin{align}
\int_{B_{\rho _{0}}}V\,dx& \leq |B_{\rho _{0}}|^{1-\frac{1}{\nu }}\left(
\int_{B_{\rho _{0}}}V^{\nu }\,dx\right) ^{\frac{1}{\nu }}  \notag
\label{50M} \\
& \leq c\left( \frac{1}{(R_{0}-\rho _{0})^{2}}\right) ^{\frac{2^{\ast }\mu }{%
(2\mu -2^{\ast })\nu }}\left( \int_{B_{R_{0}}}(1+f(Du))\,dx\right) ^{\frac{%
2^{\ast }(\mu -1)}{(2\mu -2^{\ast })\nu }}.
\end{align}%
Therefore, with the notation $\theta _{0}=2\frac{2^{\ast }\mu }{(2\mu
-2^{\ast })\nu }$, there exists a constant $c=c(n,\alpha ,\beta ,\gamma )$
such that 
\begin{equation}
\int_{B_{\rho _{0}}}\left\{ 1+|Du|^{2\gamma }\left( g_{2}(|Du|)\right)
^{2\gamma -1}\right\} \,dx  \label{51M}
\end{equation}%
\begin{equation*}
\leq \frac{c}{(R_{0}-\rho _{0})^{\theta _{0}}}\left( \int_{B_{R_{0}}}\left\{
1+f(Du)\right\} \,dx\right) ^{\frac{2^{\ast }(\mu -1)}{(2\mu -2^{\ast })\nu }%
}.
\end{equation*}%
By (\ref{43M}) finally we deduce that for all $\rho _{0}$, $R_{0}$ with $%
0<\rho _{0}<R_{0}$ 
\begin{equation}
\left\Vert Du\right\Vert _{L^{\infty }\left( B_{\rho _{0}};\mathbb{R}%
^{n}\right) }^{2}\leq \frac{c}{\left( R_{0}-\rho _{0}\right) ^{\theta _{2}}}%
\left( \int_{B_{R_{0}}}\left\{ 1+f\left( Du\right) \right\} \,dx\right)
^{\theta _{1}},  \label{final estimate}
\end{equation}%
where we have used the notation 
\begin{equation}
\theta _{1}=\tfrac{2^{\ast }(\mu -1)}{(2\mu -2^{\ast })\nu }\cdot \tfrac{%
2^{\ast }-2}{2^{\ast }-\alpha -2(\gamma -1)}\;\;\;\text{and}\;\;\;\theta
_{2}=\theta _{0}\tfrac{2^{\ast }-2}{2^{\ast }-\alpha -2(\gamma -1)}\,.
\label{52M}
\end{equation}%
We observe that $\theta _{1}>1$; indeed $\tfrac{2^{\ast }-2}{2^{\ast
}-\alpha -2(\gamma -1)}\geq 1$ since $\alpha \geq 2$ and $\gamma \geq 1$.
The other factor of $\theta _{1}$ is strictly greater than $1$. In fact,
being $\alpha \geq 2$ and $\gamma \geq 1$, we have that $\nu <\tfrac{2^{\ast
}}{\alpha -2+2\gamma }\leq \tfrac{2^{\ast }}{2}.$ Then we deduce 
\begin{equation}
\tfrac{2^{\ast }(\mu -1)}{(2\mu -2^{\ast })\nu }>\tfrac{2(\mu -1)}{2\mu
-2^{\ast }}>1\,.  \label{theta3}
\end{equation}

\subsection{Step 8. $W^{2,2}-$estimate}

We are ready to prove the $W^{2,2}-$estimate 
\eqref{bound on the second
derivatives}. We start from (\ref{25M}). At the beginning of our proof we
choose the test function $\varphi =\eta ^{2}u_{x_{k}}\Phi \left( \left\vert
Du\right\vert \right) $, with $\eta \in C^{1}\left( \Omega \right) $ with
compact support in $\Omega $, $\Phi :\left[ 0,+\infty \right) \rightarrow %
\left[ 0,+\infty \right) $ is a generic nonnegative, increasing, locally
Lipschitz continuous function in $\left[ 0,+\infty \right) $. Here, in the
estimate in (\ref{25M}), we simply consider the case when $\Phi $ is the
constant identically equal to $1$. We obtain 
\begin{align*}
\int_{\Omega }\eta ^{2}g_{1}\left( \left\vert Du\right\vert \right)
\,\left\vert D^{2}u\right\vert ^{2}\,dx\leq & 8Mn\int_{\Omega }\eta
\left\vert D\eta \right\vert \left\vert Du\right\vert ^{1+\gamma }\left(
g_{2}\left( \left\vert Du\right\vert \right) \right) ^{\gamma }\,dx\, \\
& +8(Mn)^{2}\int_{\Omega }\eta ^{2}\left( 1+\left\vert Du\right\vert
^{\gamma }\right) ^{2}\left( g_{2}\left( \left\vert Du\right\vert \right)
\right) ^{2\gamma -1}\,dx \\
& +8\int_{\Omega }\left\vert D\eta \right\vert ^{2}g_{2}\left( \left\vert
Du\right\vert \right) \left\vert Du\right\vert ^{2}\,dx.
\end{align*}%
Since $\gamma \geq 1$, in the integrands in the right hand side we can use
the inequalities for a constant $c\geq 1$ 
\begin{equation*}
\left\{ 
\begin{array}{l}
\left\vert Du\right\vert ^{1+\gamma }\text{,}\;\;\left( 1+\left\vert
Du\right\vert ^{\gamma }\right) ^{2}\;\;\text{and}\;\;\left\vert
Du\right\vert ^{2}\leq c\left( 1+\left\vert Du\right\vert ^{2\gamma }\right)
\\ 
\left( g_{2}\left( \left\vert Du\right\vert \right) \right) ^{\gamma }\;\;%
\text{and}\;\;g_{2}\left( \left\vert Du\right\vert \right) \leq 1+\left(
g_{2}\left( \left\vert Du\right\vert \right) \right) ^{2\gamma -1} \\ 
\eta \left\vert D\eta \right\vert \text{,}\;\;\eta ^{2}\;\;\text{and}%
\;\;\left\vert D\eta \right\vert ^{2}\leq \left( \eta +\left\vert D\eta
\right\vert \right) ^{2}%
\end{array}%
\right.
\end{equation*}%
and, for a new positive constant $c$, we get 
\begin{align}
\int_{\Omega }& \eta ^{2}g_{1}\left( \left\vert Du\right\vert \right)
\,\left\vert D^{2}u\right\vert ^{2}\,dx  \notag
\label{second derivatives estimate 2} \\
& \leq c\int_{\Omega }\left( \eta +\left\vert D\eta \right\vert \right)
^{2}\left( 1+\left\vert Du\right\vert ^{2\gamma }\right) \left( 1+\left(
g_{2}\left( \left\vert Du\right\vert \right) \right) ^{2\gamma -1}\right)
\,dx\,.
\end{align}%
The function $g_{2}:\left[ 0,+\infty \right) \rightarrow \left[ 0,+\infty
\right) $ is nonnegative, increasing and, as in \eqref{g(1)}, $g_{2}(t)\geq
g_{2}(1)\geq 1$ for all $t\geq 1$; then 
\begin{equation*}
\left\vert Du\right\vert ^{2\gamma }\;\;\text{and}\;\;\left( g_{2}\left(
\left\vert Du\right\vert \right) \right) ^{2\gamma -1}\leq 1+\left\vert
Du\right\vert ^{2\gamma }\left( g_{2}\left( \left\vert Du\right\vert \right)
\right) ^{2\gamma -1}.
\end{equation*}%
Thus from (\ref{second derivatives estimate 2}), for a further constants $%
c^{\prime }$ we deduce 
\begin{align}
\int_{\Omega }& \eta ^{2}g_{1}\left( \left\vert Du\right\vert \right)
\,\left\vert D^{2}u\right\vert ^{2}\,dx  \notag
\label{second derivatives estimate 3} \\
& \leq c^{\prime }\int_{\Omega }\left( \eta +\left\vert D\eta \right\vert
\right) ^{2}\left( 1+\left\vert Du\right\vert ^{2\gamma }\left( g_{2}\left(
\left\vert Du\right\vert \right) \right) ^{2\gamma -1}\right) \,dx\,.
\end{align}%
We consider concentric balls $B_{R}$, $B_{\left( R+\rho \right) /2}$ and $%
B_{\rho }$ compactly contained in $\Omega $, with $0<\rho <\frac{R+\rho }{2}%
<R<R_{1}$. As usual we also consider a test function $\eta \in
C_{0}^{1}\left( B_{R}\right) $, $0\leq \eta \leq 1$ in $B_{R}$, $\eta =1$ in 
$B_{\rho }$ and $\eta =0$ out of $B_{\left( R+\rho \right) /2}$, with
pointwise gradient bound $\left\vert D\eta \right\vert \leq 4/\left( R-\rho
\right) $ in $B_{\left( R+\rho \right) /2}\backslash B_{\rho }$ and of
course in all $B_{R}$ too. By (\ref{second derivatives estimate 3}) we have 
\begin{equation*}
\int_{B_{\rho }}g_{1}\left( \left\vert Du\right\vert \right) \,\left\vert
D^{2}u\right\vert ^{2}\,dx\leq c^{\prime }\left( 1+\tfrac{4}{R-\rho }\right)
^{2}\int_{B_{\frac{R+\rho }{2}}}\left( 1+\left\vert Du\right\vert ^{2\gamma
}\left( g_{2}\left( \left\vert Du\right\vert \right) \right) ^{2\gamma
-1}\right) \,dx\,.
\end{equation*}%
Since $R\leq R_{1}$, then $\tfrac{4}{R-\rho }\geq \tfrac{4}{R_{1}}$ and thus 
$1\leq \tfrac{R_{1}}{R-\rho }$. We obtain 
\begin{equation}
\int_{B_{\rho }}g_{1}\left( \left\vert Du\right\vert \right) \,\left\vert
D^{2}u\right\vert ^{2}\,dx\leq \tfrac{c^{\prime \prime }\left(
R_{1}^{2}+4\right) ^{2}}{\left( R-\rho \right) ^{2}}\int_{B_{\frac{R+\rho }{2%
}}}\left( 1+\left\vert Du\right\vert ^{2\gamma }\left( g_{2}\left(
\left\vert Du\right\vert \right) \right) ^{2\gamma -1}\right) \,dx.
\label{second derivatives estimate 5}
\end{equation}%
We rewrite \eqref{51M} in the balls $B_{\left( R+\rho \right) /2}$ and $%
B_{R} $, being (with the notation there) $R_{0}-\rho _{0}=R-\frac{R+\rho }{2}%
=\frac{R-\rho }{2}$, 
\begin{equation}
\int_{B_{\frac{R+\rho }{2}}}\left[ 1+|Du|^{2\gamma }\left(
g_{2}(|Du|)\right) ^{2\gamma -1}\right] \,dx\leq \frac{2^{\theta _{0}}c}{%
(R-\rho )^{\theta _{0}}}\left( \int_{B_{R}}(1+f(Du))\,dx\right) ^{\frac{%
2^{\ast }(\mu -1)}{(2\mu -2^{\ast })\nu }}.
\label{second derivatives estimate 6}
\end{equation}%
By combining (\ref{second derivatives estimate 5}),(\ref{second derivatives
estimate 6}) we finally get 
\begin{equation*}
\int_{B_{\rho }}g_{1}\left( \left\vert Du\right\vert \right) \,\left\vert
D^{2}u\right\vert ^{2}\,dx\leq \frac{c^{\prime \prime \prime }}{(R-\rho
)^{2+\theta _{0}}}\left( \int_{B_{R}}(1+f(Du))\,dx\right) ^{\frac{2^{\ast
}(\mu -1)}{(2\mu -2^{\ast })\nu }}.
\end{equation*}%
for a constant $c^{\prime \prime \prime }$. Therefore the $W_{\mathrm{loc}%
}^{2,2}\left( \Omega \right) -$bound stated in (\ref{bound on the second
derivatives}) is obtained, with exponent $\theta _{3}=\frac{2^{\ast }(\mu -1)%
}{(2\mu -2^{\ast })\nu }$, which is greater than $1$ (see \eqref{theta3}).

\section*{Acknowledgements}

A. Nastasi and P. Marcellini are members of the \textit{Gruppo Nazionale per
l'Analisi Matematica, la Probabilit\`{a} e le loro Applicazioni} (GNAMPA) of
the Istituto Nazionale di Alta Matematica (INdAM) and they were partly
supported by GNAMPA-INdAM Project 2023 \textit{Regolarit\`{a} per problemi
ellittici e parabolici con crescite non standard}, CUP E53C22001930001.%
\newline
\noindent This research was conducted while C. Pacchiano Camacho was at the
Okinawa Institute of Science and Technology (OIST) through the Theoretical
Sciences Visiting Program (TSVP).\newline
\noindent Part of this material was obtained when P. Marcellini and C.
Pacchiano Camacho visited University of Palermo, Italy in July 2024. The
visit was partially supported by The Best Paper Award 2023 Prize Fundings
assigned by the Department of Engeneering of University of Palermo to A.
Nastasi. 

\end{document}